\documentclass[12pt,reqno, aps]{amsart}
\usepackage{graphicx,epstopdf}

\newcommand{\Lim}[1]{\raisebox{0.6ex}{\scalebox{0.9}{$\displaystyle \lim_{#1}\;$}}}

\numberwithin{equation}{section}
\newtheorem{thm}{Theorem}[section]

\newtheorem{lem}[thm]{Lemma}

\theoremstyle{definition}

\theoremstyle{remark}

\numberwithin{equation}{section}

\begin{document}

\title[Coupled Hartree type equations ]{\textsc{Existence and stability of standing waves for coupled nonlinear Hartree type equations}
}

\author[SANTOSH BHATTARAI]{ SANTOSH BHATTARAI }
\address{}
\email{sbmath55@gmail.com, bhattarais@trocaire.edu}



\maketitle

\begin{abstract}
We study existence and stability of standing waves for coupled nonlinear Hartree type equations
\[
-i\frac{\partial}{\partial t}\psi_j=\Delta \psi_j+\sum_{k=1}^m \left(W\star |\psi_k|^p \right)|\psi_j|^{p-2}\psi_j,
\]
where $\psi_j:\mathbb{R}^N\times \mathbb{R}\to \mathbb{C}$ for $j=1, \ldots, m$ and the potential $W:\mathbb{R}\to [0, \infty)$ satisfies certain assumptions.
Our method relies on a variational characterization of standing waves based on minimization of the energy when $L^2$ norms
of component waves are prescribed. We obtain existence and stability results for two and three-component systems and for a certain range of $p$.
In particular, our argument works in the case when $W(x)=|x|^{-\alpha}$ for some $\alpha>0.$

\end{abstract}

\section{Introduction}

The Pekar energy functional
\[
\mathcal{P}(\phi)=\frac{1}{2}\int_{\mathbb{R}^3}|\nabla \phi(x)|^2\ dx-\int_{\mathbb{R}^3\times \mathbb{R}^3}\frac{ |\phi(x)|^2|\phi(y)|^2}{|x-y|}\ dxdy
\]
arises from an approximation to the Hartree-Fock theory for
one component plasma as discussed in Lieb's paper \cite{Lieb1}. Here $\phi$ represents
the wave function of the electron.
For the energy functional of
the electronic wave function,
it is natural to impose the normalization constraint that ${\textstyle \int_{\mathbb{R}^3}|\phi|^2}\ dx$ be held constant.
The minimizer of the
problem of minimizing $\mathcal{P}(\phi)$ under the normalization condition solves the equation
\begin{equation}\label{choq}
\begin{aligned}
   -\Delta \phi+\lambda \phi = \left(\int_{\mathbb{R}^3}\frac{|\phi(y)|^2}{|x-y|}\ dy \right)\phi,\ \ \int_{\mathbb{R}^3}|\phi|^2\ dx =M>0,
\end{aligned}
\end{equation}
where $\lambda$ is the Lagrange multiplier.
Depending on the context of the application, the equation \eqref{choq} is also
called the Choquard equation or Schr\"{o}dinger-Newton equation.
The theory for nonlinear Choquard equation and its variants is fairly well developed in the mathematics
literature by now, though there are
still many interesting open questions. A complete survey of
available results goes  beyond  the  scope  of  this  paper; we only refer the
interested reader to \cite{SBC, Lieb1, Lio3, Lio, MaLi}.
The theory for coupled systems of such equations is much
less developed, though they, too, arise as models for a variety of physical phenomena.
Considered herein are the
coupled systems of nonlinear
Schr\"{o}dinger equations with nonlocal interaction in the form
\begin{equation}\label{NChoM}
-\Delta \phi_j+\lambda_j\phi_j=\sum_{k=1}^m \left(W\star |\phi_k|^p \right)|\phi_j|^{p-2}\phi_j\ \ \mathrm{in}\ \mathbb{R}^N,\ 1\leq j\leq m,
\end{equation}
where $\star$ denotes the convolution operator and $W:\mathbb{R}^N\to [0, \infty)$ is the
convolution potential satisfying certain assumptions (see below).
The information about the properties of the system \eqref{NChoM} does not change with the time and it is said to be in a stationary state.

\smallskip

By a solution of \eqref{NChoM} we mean a pair consisting of a function $(\phi_1, \ldots, \phi_m)$ in the space
$Y_m=(\mathrm{H}^1(\mathbb{R}^N))^m$ and $\lambda=(\lambda_1, \ldots, \lambda_m)\in \mathbb{R}^m$ solving the system \eqref{NChoM}.
(Here $\mathrm{H}^1(\mathbb{R}^N)$ denotes the $L^2$-based Sobolev space of complex-valued functions on $\mathbb{R}^N$.)
Solutions $(\phi_1, \ldots, \phi_m; \lambda)$ of \eqref{NChoM} can be obtained
as critical points of the functional
\begin{equation}
\begin{aligned}
\mathcal{I}(\phi_1,\ldots, \phi_m)=\frac{1}{2}\sum_{j=1}^m\int_{\mathbb{R}^N}|\nabla \phi_j|^2\ dx
-\frac{1}{2p}\sum_{k, j=1}^m\int_{\mathbb{R}^N}\left(W\star|\phi_k|^p\right)|\phi_j|^p \ dx
\end{aligned}
\end{equation}
subject to the constraints that ${\textstyle \int_{\mathbb{R}^N}|\phi_j|^2}\ dx,\ 1\leq j\leq m,$ be held constants.
In other words, the nonlocal Schr\"{o}dinger system \eqref{NChoM}
arises as the Euler-Lagrange
equations for the
problem of finding
\begin{equation}
I_{M_1, \ldots, M_m}^{(m)}=\inf\left\{\mathcal{I}(\phi): \phi=(\phi_1,\ldots, \phi_m)\in Y_m,\ \int_{\mathbb{R}^N}|\phi_j|^2\ dx=M_j,\ 1\leq j\leq m \right\}.
\end{equation}
The unknown $\lambda_j$ in the system \eqref{NChoM} appear as Lagrange multipliers.
Given any solution $(\phi_1, \ldots, \phi_m; \lambda)$ of \eqref{NChoM}, the functions $\psi_j:\mathbb{R}^N\times (0,\infty)\to \mathbb{C}$
defined by $
\psi_j(x,t)=e^{-i\lambda_jt}\phi_j(x)
$
depends on the time explicitly and
the wave function $(\psi_1, \ldots, \psi_m)$
is called a standing wave for time-dependent Schr\"{o}dinger system with nonlocal nonlinearities
\begin{equation}\label{NChoM2}
-i\frac{\partial}{\partial t}\psi_j=\Delta \psi_j+\sum_{k=1}^m \left(W\star |\psi_k|^p \right)|\psi_j|^{p-2}\psi_j,\ 1\leq j\leq m\ .
\end{equation}
Systems of the form \eqref{NChoM2} are also called nonlinear Hartree like systems.
Motivation for the theoretical studies of coupled nonlinear Schr\"{o}dinger equations or Hartree equations
comes with the recent remarkable
experimental advances in multi-component Bose-Einstein condensates (\cite{Anderso}).
As pointed out in (\cite{Lieb7, Lio8}),
nonlinear Hartree type systems with the Coulomb potential $W(x)=|x|^{-1}$ are also used
as models to describe the interaction between electrons in the Hartree-Fock theory in Quantum Chemistry.
The interaction between electrons is said to be repulsive (resp. attractive) when the sign in front of
the interaction
terms in the Hamiltonian is positive (resp. negative). Systems of the form considered in this paper arise
as models for a variety
of physical situations in which quantum particles interactive attractively. Examples include boson stars,
systems of polarons in a lattice, and some Bose gases.
For a discussion of how the Hartree type equation appears as a mean-field limit for many-particle
boson systems, the reader may consult \cite{2Ref1,2Ref2, 2Ref3}.
The two-component nonlinear Hartree type systems with $W(x)=\delta(x)$ (the delta function) has applications especially in nonlinear optics (\cite{Me1a, Me1b}).
Nonlocal nonlinearities
have attracted considerable interest as means of eliminating collapse and
stabilizing multidimensional solitary waves, as was shown in the context of optics (\cite{Bang9}). It appears naturally in optical systems (\cite{Lit1})
and is also known to influence the propagation of electromagnetic waves in plasmas (\cite{Ber5}).
In the theory of Bose-Einstein condensation, nonlocality accounts for the finite-range many-body interaction (\cite{Dal2}).

\smallskip

The purpose of this paper is twofold. First, we prove the precompactness of minimizing sequences for two-parameter variational
problem $I_{M_1, M_2}^{(2)}.$ As a consequence
we obtain existence and stability of two-parameter family of standing waves for coupled nonlinear Hartree equations.
Another purpose of this paper is to generalize the arguments to establish the precompactness of minimizing sequences
for the three-parameter problem $I_{M_1, M_2, M_3}^{(3)}$.
This leads to results concerning existence and stability
of true three-parameter family of standing waves for coupled
nonlinear Hartree equations. To our knowledge, this is the first paper which
establishes existence and stability of standing waves for 3-coupled Hartree type systems under three independent normalization constraints.

\smallskip

The key to our analysis is the concentration compactness lemma of P. L. Lions (Lemma~I.1 of \cite{Lio}).
For single nonlinear dispersive evolution equations in which the
variational problems characterizing standing waves take the form
\[
\mathrm{minimize}\  \mathcal{A}(u)=\int_{\mathbb{R}^N}A(u(x), \nabla u(x))\ dx\ \ \mathrm{s.t.}\ \int_{\mathbb{R}^N}|u|^2\ dx=M>0,
\]
the concentration
compactness technique is widely used for proving the relative compactness of
minimizing sequences (and hence the stability
of the set of minimizers provided that both the energy $\mathcal{A}$ and the mass functional are conserved by the flow associated to the evolution equation, see \cite{CLi}).
Quite differently from the one-parameter case, its application for showing
the relative compactness of minimizing sequences of variational problems under two or more constraint parameters,
however, seems to be more
complicated.
In particular, putting the method
into practice requires ruling out the case which Lions called \textit{dichotomy} by establishing
certain strict inequality for the function of constraint parameter(s). For one-parameter variational
problems, as stated in Lions' paper \cite{Lio}, preventing dichotomy is equivalent to
verifying the strict inequality in the form
\begin{equation}\label{ST1}
I_{M}<I_{T}+I_{M-T},\ \forall T\in [0, M),
\end{equation}
where $I_{M}$ denotes the infimum of $\mathcal{A}$ over $\left\{u\in H^1(\mathbb{R}^N):\int_{\mathbb{R}^N}|u|^2\ dx=M \right\}.$
In \cite{[A]}, J. Albert has illustrated the method by proving the strict inequality in a slightly different form
\begin{equation}\label{ST2}
I_{M_1+M_2}<I_{M_1}+ I_{M_2},\ \forall M_1, M_2>0.
\end{equation}

\smallskip

More recently, the method of preventing dichotomy of minimizing sequences for
two-parameter variational problems was developed in \cite{[AB11]} (see also \cite{[SB3]}).
In order to employ strategies of \cite{[AB11]}
for the problem $I_{M_1, M_2}^{(2)}$,
one requires to verify the strict inequality
\begin{equation}\label{ST3}
I_{M_1+T_1, M_2+T_2}^{(2)}<I_{M_1, M_2}^{(2)}+I_{T_1, T_2}^{(2)}
\end{equation}
for all $M=(M_1, M_2), T=(T_1, T_2)\in \mathbb{R}_{+}^2\cup \{\mathbf{0}\}$ satisfying
$M, T\neq \{\mathbf{0}\}$ and $M+T\in \mathbb{R}_{+}^2.$ (Here $\mathbb{R}_{+}$ denotes the interval
$(0, \infty)$ and $\mathbb{R}_{+}^2=\mathbb{R}_{+} \times \mathbb{R}_{+}.$)
While several techniques are available to prove the strict inequality for one-parameter problems,
the proof of strict inequality for two-parameter problems such as \eqref{ST3}, even for the most universal
choice of coupling terms, is much less understood.
Furthermore, when one generalizes the strict inequality \eqref{ST3} for $m$-parameter problem $I_{M_1, \ldots, M_m}^{(m)},$ it
takes the form
\begin{equation}\label{V2def}
I_{M_1+T_1, \ldots, M_m+T_m}^{(m)}<I_{M_1, \ldots, M_m}^{(m)}+I_{T_1, \ldots, T_m}^{(m)}
\end{equation}
and one requires to verify \eqref{V2def} for all possible cases based on the values
\[
M=(M_1,\ldots, M_m), T=(T_1, \ldots, T_m)\in \mathbb{R}_{+}^m\cup \{\mathbf{0}\},\
M, T\neq \{\mathbf{0}\}\ M+T\in \mathbb{R}_{+}^m.
\]
This makes the situation even more complicated for $m$-parameter problems and
the problem of employing the machinery of compactness by concentration under multiple constraints
remains widely open.
The task of proving the strict inequalities for $I_{M_1,M_2}^{(2)}$ and the three-parameter problem $I_{M_1,M_2, M_3}^{(3)}$, and preventing
dichotomy of minimizing sequences
will occupy us through most of
Sections~\ref{section2} and \ref{section3}.

\smallskip

\noindent For any $1\leq r<\infty$, we denote by $L_w^r(\mathbb{R}^N)$ (the weak $L^r$ space) the set of all
measurable functions $f:\mathbb{R}^N\to \mathbb{C}$
such that
\[
\|f\|_{L_w^r}=\sup_{M>0}M|\{ x:|f(x)|>M\}|^{1/r}<\infty.
\]
Throughout the paper, we require the power $p$ and the convolution potential $W\in L_w^{r}(\mathbb{R}^N)$ to satisfy the following assumptions
\begin{itemize}
\item[(h0)] The power $p$ satisfies
\[
2\leq p< \frac{2r-1 }{r }+\frac{2}{N}\ \ \mathrm{with}\ \  \frac{1}{r}<\frac{2}{N}.
\]
\item[(h1)] The potential $W:\mathbb{R}^N\to [0, \infty)$ is radially symmetric i.e., $W(x)=W(|x|)$, and satisfies $W(r)\to 0$ as $r\to \infty.$
\item[(h2)] There exists $\Gamma$ satisfying $\Gamma<2+2N-pN$ such that
\[
W(\theta \xi)\geq \theta^{-\Gamma}W(\xi)\ \ \mathrm{for\ any}\ \theta>1.
\]
\end{itemize}
The results in this paper hold for the Coulomb type potential $W(x)=|x|^{-\alpha}$ for some $\alpha>0.$
Our main results are as follows:
\begin{thm}\label{exisTHM}
Suppose $m=2,3$ and the assumptions (h0), (h1), and (h2) hold.
For every $M=(M_1, \ldots, M_m)\in \mathbb{R}_{+}^m,$ define
\[
\Lambda^{(m)}(M)=\left\{\phi=(\phi_1, \ldots, \phi_m)\in Y_m:\mathcal{I}(\phi)=I_M^{(m)},\ \|\phi_j\|_{L^2}^2=M_j,\ 1\leq j\leq m \right\}.
\]
The following statements hold:

\smallskip

$(a)$ For every $M=(M_1, \ldots, M_m)\in \mathbb{R}_{+}^m,$ there exists a nonempty set $\Lambda^{(m)}(M)\subset Y_m$ such that
for every $\phi\in \Lambda^{(m)}(M)$, there exists $(\lambda_1, \ldots, \lambda_m)$ such that $\psi_j(x,t)=e^{-i\lambda_jt}\phi_j$
is a standing wave for \eqref{NChoM2} satisfying $\int_{\mathbb{R}^N}|\phi_j|^2\ dx=M_j,\ 1\leq j\leq m.$

\smallskip

$(b)$ For every complex-valued minimizer $\phi$ of $I_{M}^{(m)},$ there exists $\theta_j\in \mathbb{R}$ and real-valued functions
$\widetilde{\phi}_j$ such that
\[
\widetilde{\phi}_j(x)>0\ \ \mathrm{and}\ \ \phi_j(x)=e^{i\theta_j}\widetilde{\phi}_j(x),\ \ \forall x\in \mathbb{R}^N,\ 1\leq j\leq m.
\]
\end{thm}

\smallskip

We recall here that for the initial-value problem for \eqref{NChoM2} to be (local) well-posed, its solution $\psi(x,t)=(\psi_1(x,t), \ldots, \psi_m(x,t))$ should exist for some $T>0$ for arbitrary choices of the initial data $\psi(x,0)=(\psi_1(x,0), \ldots, \psi_m(x,0))$ in the function class $Y_m,$
and the solution should be unique and depend continuously
on the initial data. In the next result, we assume that the initial-value
problem for \eqref{NChoM2} satisfies the well-posedness property.
Moreover, the following conservation laws hold:
\[
\mathcal{I}(\psi(\cdot, t))=\mathcal{I}(\psi(\cdot, 0));\ \int_{\mathbb{R}^N}|\psi_j(x,t)|^2\ dx=\int_{\mathbb{R}^N}|\psi_{j}(x, 0)|^2\ dx,\ 1\leq j \leq m.
\]
\begin{thm}\label{StaTHM}
Under the same hypotheses as in Theorem~\ref{exisTHM}, the set $\Lambda^{(m)}(M)$ is stable for the
associated initial-value problem of \eqref{NChoM2}, i.e.,
for every $\varepsilon>0,$ there exists $\delta(\varepsilon)>0$ such that whenever $(\psi_{01}, \ldots,\psi_{0m})\in Y_m$ satisfies
\[
\inf_{\phi\in \Lambda^{(m)}(M)}\|(\psi_{01}, \ldots,\psi_{0m})-\phi\|_{Y_m}\leq \delta(\varepsilon),
\]
then any solution $\psi(\cdot, t)=(\psi_1(\cdot, t), \ldots, \psi_m(\cdot, t))$ of \eqref{NChoM2} with initial datum $\psi_j(\cdot,0)=\psi_{0j}$ satisfies
\[
\sup_{t\geq 0}\inf_{\phi\in \Lambda^{(m)}(M)}\|\psi(t,\cdot)-\phi\|_{Y_m}< \varepsilon.
\]
\end{thm}

\section{The variational problem}
In this section, we prove number of lemmas which are needed in the sequel to prove our main results.
Throughout this section we do not distinguish the case $m=2$ and $m=3.$ The results of this section remain hold for an arbitrary $m.$

\smallskip

In what follows,
for $s>0,$ we denote by $\Sigma_s$ the sphere
\[
\Sigma_s=\left\{f\in H^1(\mathbb{R}^N):\int_{\mathbb{R}^N}|f|^2\ dx=s \right\}.
\]
We always denote $m$-tuples in $\mathbb{R}_{+}^m$ as
$M=(M_1, \ldots, M_m)$, $T=(T_1, \ldots, T_m)$, etc.
For any $M\in \mathbb{R}_{+}^m,$ we write
$
\Sigma_M^{(m)}=\Sigma_{M_1}\times \ldots \times \Sigma_{M_m}.
$
To avoid tedious expressions, we often write
\[
Q(f,g)=|f(x)|^p|g(y)|^p\ \  \mathrm{for}\ x,y\in \mathbb{R}^N
\]
and for any $q>0,$ we shall denote the Coulomb-type potential by
\begin{equation}\label{QdefA}
\mathbb{F}_q(f,g)=\frac{1}{q}\int_{\mathbb{R}^N\times \mathbb{R}^N}W(|x-y|)Q(f,g)\ dxdy.
\end{equation}
We will make use of the following Hardy-Littlewood-Sobolev inequality.
\begin{lem}
For every $f\in L^q(\mathbb{R}^N), g\in L_w^{r}(\mathbb{R}^N),$ and $h\in L^t(\mathbb{R}^N)$ with $1<q, r, t<\infty$ and
${\textstyle \frac{1}{q}+\frac{1}{r}+\frac{1}{t}=2},$ there exists $C=C(q, N, r)>0$ such that
\[
\left\vert\int_{\mathbb{R}^N\times \mathbb{R}^N}f(x)g(x-y)h(x)\ dxdy \right\vert\leq C\|f\|_{L^q}\|g\|_{L_w^{r}}\|h\|_{L^t}.
\]
\end{lem}
\begin{proof}
See Lieb and Loss, Analysis \cite{Lieb}.
\end{proof}
In what follows we use the Sobolev interpolation inequalities
\begin{equation}\label{GNinq}
\left(\int_{\mathbb{R}^N}|u|^s\ dx \right)^{2/s}\leq C(s, N)\left(\int_{\mathbb{R}^N}|\nabla u|^2\ dx \right)^\theta \left( \int_{\mathbb{R}^N}|u|^2\ dx \right)^{1-\theta}
\end{equation}
holds for every $u\in H^1(\mathbb{R}^N)$ and $s$ such that $2\leq s\leq \infty$
if $N=1,$ $2\leq s <\infty$ if $N=2,$ and $2\leq s \leq 2N/(N-2)$ if $N\geq 3,$ where $C(s, N)>0$ and $\theta$ satisfies
\[
\frac{N}{s}=\frac{\theta(N-2)}{2}+\frac{(1-\theta)N}{2}.
\]
The following lemma shows that
$I_M^{(m)}$ is well posed and minimizing sequences are uniformly bounded in $Y_m.$
\begin{lem}\label{NegIM2}
For $M=(M_1, \ldots, M_m)\in \mathbb{R}_{+}^m,$ let $\{(u_1^n, \ldots, u_m^n)\}_{n\geq 1}$ be any sequence in $Y_m$ satisfying
\[
\mathcal{I}(u_1^n, \ldots, u_m^n)\to I_M^{(m)}\ \ \mathrm{and}\ \ \lim_{n\to \infty}\|u_j^n\|_{L^2}^2=M_j,\ 1\leq j\leq m.
\]
Then there exists a constant $C>0$ such that
$
\sum_{j=1}^m\|u_j^n\|_{H^1(\mathbb{R}^N)}^2\leq C
$
for all $n.$ Moreover, for every $M\in \mathbb{R}_{+}^{m},$ one has
\[
-\infty<I_M^{(m)}<0.
\]
\end{lem}
\begin{proof}
We begin with the following observation. In view of the Hardy-Littlewood-Sobolev inequality, the integral
\[
\mathbb{F}_q(f,g)=\frac{1}{q}\int_{\mathbb{R}^N\times \mathbb{R}^N}W(x-y)|f(x)|^p|g(y)|^p\ dxdy,\ q>0,
\]
is well-defined if $|f|^p, |g|^p\in L^t(\mathbb{R}^N)$ for all $t>1$ satisfying the condition
\[
\frac{1}{t}+\frac{1}{r}+\frac{1}{t}=2,\ \ \mathrm{or}\ \ t=\frac{2r}{2r-1}.
\]
By our assumption, we have that
\[
\frac{1}{tp}=\frac{2r-1}{2pr}=\frac{1}{p}-\frac{1}{2pr}>\frac{1}{p}-\frac{2N+2-pN}{2Np}=\frac{1}{2}-\frac{1}{Np}\geq \frac{N-2}{2N}.
\]
It follows that $|f|^p\in L^{\textstyle \frac{2r}{2r-1}}(\mathbb{R}^N)$ for every $f\in \mathrm{H}^1(\mathbb{R}^N).$
Using the Hardy-Littlehood-Sobolev inequality and the Gagliardo-Nirenberg inequality, we obtain that
\begin{equation}
\begin{aligned}\label{ESTa}
\mathbb{F}_q(u_j^n, u_k^n) = \frac{1}{q}\int_{\mathbb{R}^N}(W\star|u_k^n|^p)&|u_j^n|^p\ dx
 \leq C\|W\|_{L_w^{r}}\|u_j^n\|_{L^{\frac{2pr}{2r-1} }}^p\|u_k^n\|_{L^{\frac{2pr}{2r-1} }}^p\\
& \leq C \|u_j^n\|_{L^2}^{(1-\mu)p}\|\nabla u_j^n\|_{L^2}^{\mu p}\|u_k^n\|_{L^2}^{(1-\mu)p}\|\nabla u_k^n\|_{L^2}^{\mu p}\\
& \leq C \|\nabla u_j^n\|_{L^2}^{\mu p}\|\nabla u_k^n\|_{L^2}^{\mu p},
\end{aligned}
\end{equation}
where $\mu=(Nrp-2Nr+N)/2rp.$
To show that $\{(u_1^n, \ldots, u_m^n)\}_{n\geq 1}$ is bounded, using the estimate \eqref{ESTa} and the fact that the sequence $\{\mathcal{I}(u_1^n, \ldots, u_m^n)\}_{n\geq 1}$ is bounded in $\mathbb{R},$ we obtain
\[
\begin{aligned}
\frac{1}{2}\|(u_1^n, \ldots, u_m^n)\|_{Y_m}^2& =\mathcal{I}(u_1^n, \ldots, u_m^n)+\sum_{k, j=1}^m\mathbb{F}_{2p}(u_k^n,u_j^n)
+\frac{1}{2}\sum_{j=1}^m\|u_j^n\|_{L^2}^2 \\
&\leq C(N, r, p, M) \left( 1+\|(u_1^n, \ldots, u_m^n)\|_{Y_m}^{2\mu p}\right).
\end{aligned}
\]
By the assumption (h0), we have $2\mu p = Nrp-2Nr+N<2.$ Then it follows that $\{(u_1^n, \ldots, u_m^n)\}_{n\geq 1}$ is bounded in $Y_m.$
The proof that $I^{(m)}_M>-\infty$ is immediate from \eqref{ESTa} and we omit the details.

\smallskip

We next prove that $I^{(m)}_M<0$.
It is enough to show that there exists $(\widetilde{u}_1, \ldots,\widetilde{u}_m)\in Y_m$
such that $(\widetilde{u}_1, \ldots,\widetilde{u}_m)\in \Sigma_{M}^{(m)}$ and $\mathcal{I}(\widetilde{u}_1, \ldots,\widetilde{u}_m)<0.$
Start by picking $u_1\in \Sigma_{M_1}$ and define $u_j=(M_j/M_1)^{1/2}u_1$ for $2\leq j\leq m.$
Consider the functions $u_j^\theta:\mathbb{R}^N\to \mathbb{R}$ defined by
$
u_j^\theta(x)=\theta^{N/2}u_j(\theta x)$ for $1\leq j\leq m.$
Then one has $\widetilde{u}=(u_1^\theta, \ldots, u_m^\theta)\in \Sigma_M^{(m)}$ and for every $0<\theta<1,$ we compute
\[
\begin{aligned}
\int_{\mathbb{R}^{N}}(W\star|u_k^\theta|^p |u_j^\theta|^p\ dx &=\int_{\mathbb{R}^{N}\times \mathbb{R}^{N}} \theta^{Np} W(|x-y|)|u_j(\theta x)|^p|u_k(\theta y)|^p\ dxdy \\
& = \int_{\mathbb{R}^{N}\times \mathbb{R}^{N}} \theta^{Np} W(\theta^{-1}|\theta x-\theta y|)|u_j(\theta x)|^p|u_k(\theta y)|^p\ dxdy\\
 & \geq \int_{\mathbb{R}^{N}\times \mathbb{R}^{N}}\theta^{Np-2N+\Gamma}W(|x-y|)|u_k(y)|^p|u_j(x)|^p\ dxdy,
\end{aligned}
\]
where in the last inequality we used the assumption (h2).
Using this estimate, a direct computation yields
\begin{equation}\label{negEST}
\begin{aligned}
\mathcal{I}(u_1^\theta, \ldots, u_m^\theta)& \leq \frac{\theta^2}{2}\sum_{j=1}^m\|\nabla u_j\|_{L^2}^2-\theta^{Np-2N+\Gamma } \sum_{k,j=1}^m \int_{\mathbb{R}^{N}}(W\star|u_k|^p)|u_j|^p\ dx\\
&=\frac{\theta^2}{2}\sum_{j=1}^m\|\nabla u_j\|_{L^2}^2-\Omega\theta^{Np-2N+\Gamma }  \int_{\mathbb{R}^{N}}(W\star|u_1|^p)|u_1|^p\ dx,
\end{aligned}
\end{equation}
where the number $\Omega=\Omega(M_1, \ldots, M_m, p)>0$ is given by
\[
\Omega=\frac{1}{2p}+\frac{1}{p}\sum_{j=2}^m \left(\frac{M_j}{M_1}\right)^{p/2}+\frac{1}{2p}\sum_{k, j=2}^m\left(\frac{M_j}{M_1} \right)^{p/2}\left(\frac{M_k}{M_1} \right)^{p/2}>0.
\]
By our assumption $Np-2N+\Gamma<2,$ it follows from \eqref{negEST} that $\mathcal{I}(u_1^\theta, \ldots, u_m^\theta)<0$ for sufficiently small $\theta$ and
consequently, we get $I^{(m)}_M<0.$
\end{proof}

\begin{lem}\label{tech1}
Define the functional $E:H^1(\mathbb{R}^N)\to \mathbb{R}$ as follows
\[
E(h) =\frac{1}{2}\|\nabla h \|_{L^2}^2- \frac{1}{2p} \int_{\mathbb{R}^N\times \mathbb{R}^N}W(|x-y))Q(h,h)\ dxdy.
\]
Let $M\in \mathbb{R}_{\geq 0}^m$ be such that $M_1+\ldots+M_m>0$ and suppose
that $\{(u_1^n, \ldots, u_m^n)\}_{n\geq 1}$ be any minimizing sequence for $I_M^{(m)}.$
Then for each $j$ with $M_j>0$ and
any number $\Gamma>1,$ there exists $\delta>0$ (independent of $n$) such that for sufficiently large $n,$
\begin{equation}\label{stricIn}
E\left(\Gamma^{1/2} u_j^n\right)\leq \Gamma E\left(u_j^n\right)-\delta.
\end{equation}
\end{lem}
\begin{proof}
 We claim that for any minimizing
 sequence $\{(u_1^n, \ldots, u_m^n)\}_{n\geq 1}$ of the problem $I_M^{(m)}$ there exists $\delta>0$ and $n_0=n_0(\delta)\in \mathbb{N}$ such that
 \[
\int_{\mathbb{R}^N}|u_j^n|^{\frac{2pr}{2r-1} }\ dx\geq \delta,\ \forall n\geq n_0,
\]
provided that $M_j>0.$
To see this, suppose to the contrary that there exists some
minimizing sequence $\{(\widetilde{u}_1^n, \ldots, \widetilde{u}_m^n)\}_{n\geq 1}$ of $I_M^{(m)}$ such
that
\[
\liminf_{n\to \infty}\int_{\mathbb{R}^N}|\widetilde{u}_j^n|^{\frac{2pr}{2r-1}}\ dx=0.
\]
Using the Hardy-Littlewood-Sobolev inequality, we obtain that for any $q>0$,
\[
\begin{aligned}
\mathbb{F}_q(\widetilde{u}_k^n,\widetilde{u}_j^n)&=\frac{1}{q}\int_{\mathbb{R}^N\times \mathbb{R}^N}W(x-y)Q(\widetilde{u}_k^n,\widetilde{u}_j^n)\ dxdy \\
&\leq C\|W\|_{L_w^r}\|\widetilde{u}_k^n \|_{L^{\frac{2pr}{2r-1} } }^p\|\widetilde{u}_j^n \|_{L^{\frac{2pr}{2r-1} } }^p \\
& \leq C \|\widetilde{u}_k^n \|_{L^{\frac{2pr}{2r-1} } }^p\|\widetilde{u}_j^n \|_{L^{\frac{2pr}{2r-1} } }^p\rightarrow 0
\end{aligned}
\]
as $n\to \infty.$ Then it follows that
\[
I_M^{(m)}=\lim_{n\to \infty}\mathcal{I}(\widetilde{u}_1^n, \ldots, \widetilde{u}_m^n)\geq \liminf_{n\to \infty}\frac{1}{2}\sum_{j=1}^m\|\nabla \widetilde{u}_j^n\|_{L^2}^2\geq 0,
\]
which is a contradiction and hence the claim follows. To see \eqref{stricIn}, it follows from the Hardy-Littlewood-Sobolev inequality that
\begin{equation}\label{stricIn2}
\begin{aligned}
E\left(\Gamma^{1/2} u_j^n\right)&=\Gamma E\left(u_j^n\right)+(\Gamma-\Gamma^p)\mathbb{F}_{2p}(u_j^n, u_j^n)\\
&\leq \Gamma E\left(u_j^n\right)+ C (\Gamma-\Gamma^p)\|u_j^n \|_{L^{\frac{2pr}{2r-1} } }^{2p}
\end{aligned}
\end{equation}
Since $\Gamma>1,$ $p>2,$ and $\|u_j^n \|_{L^{\frac{2pr}{2r-1} } }\geq \delta$ for sufficiently large $n$,
the desired
inequality follows from \eqref{stricIn2}.
\end{proof}
We will need the following result concerning the
existence of positive solutions for the functional $E(u).$
\begin{lem}\label{tech3}
Suppose that the assumptions (h0), (h1), and (h1) hold.
Then for each $M>0$,
there exists a real-valued function $\phi_{M}>0$ such that
\[
E(\phi_{M})=\inf\left\{E(h):h\in \mathrm{H}^1(\mathbb{R}^N)\ \mathrm{and}\ \int_{\mathbb{R}^N}|h|^2\ dx=M \right\}.
\]
\end{lem}
\begin{proof}
This can be proven using the concentration compactness argument and a proof appears
in \cite{Lio} for the potential $W(x)=|x|^{-1}$, and in \cite{SBC} for $W:\mathbb{R}^N\to [0, \infty)$
satisfying the assumptions (h0), (h1), and (h2).
\end{proof}

\begin{lem}\label{VarE3b}
For every $M_1>0$ and $M_2>0$, let $\{(u_1^n, u_2^n)\}$ be a sequence in $H^1(\mathbb{R}^N\times H^1(\mathbb{R}^N)$
such that $\mathcal{I}(u_1^n, u_2^n)\to I_{M_1, M_2}^{(2)}$ and $\|u_j^n\|_{L^2}^2 \to M_j.$
Then there exists $\delta_j>0$ such that
for all sufficiently large $n,$
\[
E(u_1^n)-\mathbb{F}_p(u_1^n, u_2^n)\leq -\delta_1\ \mathrm{and}\ E(u_2^n)-\mathbb{F}_p(u_1^n, u_2^n)\leq -\delta_2.
\]
\end{lem}
\begin{proof}
Suppose to the contrary that there exists some minimizing sequence $\{(u_1^n, u_2^n)\}_{n\geq 1}$ of $I_{M_1, M_2}^{(2)}$ such that
\[
\liminf_{n\to \infty}\left( E(u_1^n)-\mathbb{F}_p(u_1^n, u_2^n)\right)\geq 0.
\]
Then this implies that
\begin{equation}\label{negLag2}
I_{M_1, M_2}^{(2)}=\lim_{n\to \infty}\mathcal{I}(u_1^n, u_2^n)\geq \liminf_{n\to \infty}\left(\frac{1}{2}\|\nabla u_2^n\|_{L^2}^2-\mathbb{F}_{2p}(u_2^n, u_2^n) \right).
\end{equation}
Let $\phi_{M_2}$ be as defined in Lemma~\ref{tech3} with $M=M_2$. Then it
follows from \eqref{negLag2} that $I_{M_1, M_2}^{(2)}\geq E(\phi_{M_2}).$
Next let $\psi\geq 0$ be an arbitrary function with compact support satisfying $\psi(0)=1$ and
$\|\psi\|_{L^2}^2=M_1.$ For any $\theta>0$, define $\psi_\theta(x)=\theta^{N/2}\psi(\theta x).$ Then one can show as in the
proof of Lemma~\ref{NegIM2} that for sufficiently small $\theta$,
\[
E(\psi_\theta)-\mathbb{F}_p(\psi_\theta, \phi_{M_2})<0.
\]
Thus, for this choice of $\theta,$ one obtains that
\[
I_{M_1, M_2}^{(2)}=E(\psi_\theta)-\mathbb{F}_p(\psi_\theta, \phi_{M_2})+E(\phi_{M_2})< E(\phi_{M_2}),
\]
which contradicts the fact $I_{M_1, M_2}^{(2)}\geq E(\phi_{M_2}).$ This proves that
$E(u_1^n)-\mathbb{F}_p(u_1^n, u_2^n)$ is negative for sufficiently large $n.$ The proof that $E(u_2^n)-\mathbb{F}_p(u_1^n, u_2^n)\leq -\delta_2$
goes through the same steps and we omit the details.
\end{proof}

\begin{lem}  \label{Lionslem}
Let $N\geq 1.$ Assume that $\{u_n\}_{n\geq 1}$ and $\{|\nabla u_n|\}_{n\geq 1}$ are
bounded in $L^2(\mathbb{R}^N).$
If for some $R>0,$
\begin{equation*}
\lim_{n\to \infty} \left( \sup_{y \in \mathbb{R}^N} \int_{B_{R}(y)}|u_n(x)|^{2}\ dx\right)= 0,
\end{equation*}
then the sequence $\{u_n\}_{n\geq 1}$ converges to zero in $L^q(\mathbb{R}^N)$ for any $2<q<\infty$ if $N=1,2$ and
for every $2<q<\frac{2N}{N-2}$ if $N\geq 3.$
\end{lem}
\begin{proof}
This lemma is a special case of Lions' concentration compactness lemma, see Lemma I.1 of \cite{Lio},
but for the sake of completeness we include a proof here.
Let us denote
$
\omega_n=\sup_{y\in \mathbb{R}^N}\|u_n\|_{L^2(B_R(y))}^2.
$
By assumption, we have that $\omega_n\to 0$ as $n\to \infty.$ Using the Sobolev inequalities, we obtain
\[
\|u_n\|_{L^q(B_R(y))}\leq C \|u_n\|_{L^2(B_R(y))}^{1-\lambda} \| u_n \|_{H^1(B_R(y))}^{\lambda},
\]
where $\lambda=N(q-2)/2q.$ Thus, one has that
\begin{equation}\label{vanlem3}
\begin{aligned}
\int_{B_R(y) } |u_n|^q \ dx&\leq C^q \|u_n\|_{L^2(B_R(y))}^{(1-\lambda)q} \| u_n \|_{H^1(B_R(y))}^{q\lambda}\\
&\leq C \left(\omega_n \right)^{(1-\lambda)q/2}\| u_n \|_{H^1(B_R(y))}^{q\lambda}.
\end{aligned}
\end{equation}
Now, if $q\lambda \geq 2,$ it is obvious from \eqref{vanlem3} that
\begin{equation}\label{vanlem4}
\begin{aligned}
\int_{B_R(y)} |u_n|^q & \leq C \left(\omega_n \right)^{(1-\lambda)q/2} \left(\int_{B_R(y)}\left( |\nabla u_n|^2+|u_n|^2\right)dx \right)
 \| u_n \|_{H^1}^{q\lambda-2}\\
 & \leq C \left(\omega_n \right)^{(1-\lambda)q/2}\int_{B_R(y)}\left( |\nabla u_n|^2+|u_n|^2\right)dx
 \end{aligned}
\end{equation}
Consider a countable family of balls $\{B_R(z_i)\}$ which covers $\mathbb{R}^N$ in such a way
that every vector in $\mathbb{R}^N$ belongs to at most $m+1$ balls.
Then, summing \eqref{vanlem4} over the balls $\{B_R(z_i)\}$, we obtain that
\[
\int_{\mathbb{R}^N} |u_n|^q\leq (m+1)C \left(\omega_n \right)^{(1-\lambda)q/2}\int_{\mathbb{R}^N}\left( |\nabla u_n|^2+|u_n|^2\right)dx\leq C \left(\omega_n \right)^{(1-\lambda)q/2},
\]
which gives the result for $q\lambda \geq 2,$ i.e., $q>2+\frac{4}{N}.$ Next consider the case that $q< 2+\frac{4}{N}.$ Using the H\"{o}lder inequality, we have that
\[
\|u_n\|_{L^q}^q \leq \|u_n\|_{L^2}^{2\theta}\|u_n\|_{L^{2+\frac{4}{N} }}^{(1-\theta)(2+\frac{4}{N}) },
\]
where $q=2\theta+(1-\theta)(2+\frac{4}{N})$ for some $\theta\in (0,1).$ Making use of the result for the case $q=2+\frac{4}{N},$ it follows that
$\|u_n\|_{L^q}\to 0,$ proving the lemma.
\end{proof}

Given any minimizing sequence $\{(u_1^n, \ldots, u_m^n)\}_{n\geq 1}$ of $I^{(m)}_M,$ we introduce the L\'{e}vy concentration function
\[
Q_n^{(m)}(R)=\sup_{y\in \mathbb{R}^N}\int_{B_R(y)}\left(|u_1^n|^2+ \ldots +|u_m^n|^2\right)\ dx,\ n=1,2,\ldots,
\]
where $B_R(x)\subset \mathbb{R}^N$ represents a ball with center at $x$ and radius $R.$
Then $\{Q_n^{(m)}\}$ is a sequence of nondecreasing functions on $[0, M_1+\ldots+M_m].$ By Helly's selection theorem, we can assume (up to a subsequence)
\begin{equation}\label{lamDef}
Z^{(m)} = \lim_{R\to \infty}\left(\lim_{n\to \infty}Q_n^{(m)}(R) \right)\in [0, M_1+\ldots+M_m].
\end{equation}
The case $Z^{(m)}=0$ is called the vanishing, $0<Z^{(m)}<M_1+\ldots+M_m$ is the case of dichotomy, and
$Z^{(m)}=M_1+\ldots+M_m$ is the tightness.

\begin{lem}
For any minimizing sequence $\{(u_1^n, \ldots, u_m^n)\}_{n\geq 1}$ of $I^{(m)}_M,$
the vanishing does not occur, that is, $Z^{(m)}>0.$
\end{lem}
\begin{proof}
If the vanishing does occur, then Lemma~\ref{Lionslem}
implies that ${\displaystyle \lim_{n\to \infty}\|u_j^n\|_{L^q}=0}$ for
any $\textstyle{2<q<\frac{2N}{N-2}.}$ Since $\textstyle{2<\frac{2pr}{2r-1}< \frac{2N}{N-2}}$, it follows from
the Hardy-Littlewood-Sobolev inequality that for any $t>0,$
\[
\mathbb{F}_t(u_k^n, u_j^n)=\int_{\mathbb{R}^N\times \mathbb{R}^N}W(x-y)Q(u_k^n, u_j^n)\ dxdy\to 0
\]
as $n\to \infty.$
Consequently, we have that
\[
I^{(m)}_M=\lim_{n\to \infty}\mathcal{I}(u_1^n, \ldots, u_m^n)\geq \liminf_{n\to \infty}\frac{1}{2}\sum_{j=1}^m\|\nabla u_j\|_{L^2}^2\geq 0,
\]
a contradiction and hence lemma follows.
\end{proof}

The next lemma concerns the case $Z^{(m)}=M_1+\ldots+M_m.$
\begin{lem}\label{VarE5}
Suppose that $\{(u_1^n, \ldots, u_m^n)\}_{n\geq 1}$ be any minimizing sequence for $I^{(m)}_M$ and $Z^{(m)}=M_1+\ldots+M_m.$
Then there exists $\{y_n\}\subset \mathbb{R}^N$ such that the sequence
\[
\{(u_1^n(x+y_n),\ldots,u_m^n(x+y_n) )\}_{n\geq 1},\ x\in \mathbb{R}^N,
\]
converges in $Y_m$ up to a subsequence to a function $(\phi_1, \ldots, \phi_m)\in \Lambda^{(m)}(M).$
In particular, the solution set $\Lambda^{(m)}(M)$ is nonempty.
\end{lem}
\begin{proof}
We write $\Sigma(M)=M_1+\ldots+M_m.$
Since $Z^{(m)}=\Sigma(M)$, we can find $\{y_n\}\subset \mathbb{R}^N$ such that if we write
$w_j^n=u_j^n(x+y_n), 1\leq j\leq m,$ then for every $k\in \mathbb{N},$ one can find $R_k>0$ satisfying
for sufficiently large $n,$
\begin{equation}\label{Comp1}
\int_{B_{R_k}(0)}\sum_{j=1}^m|w_j^n|^2\ dx>\Sigma(M)-\frac{1}{k}.
\end{equation}
In the sequel we denote $w_n=(w_1^n, \ldots, w_m^n).$ Since $\|w_n\|_{Y_m}\leq B$ for all $n,$ so from
Rellich-type embedding, we have that for every bounded domain $\Omega\subset \mathbb{R}^N,$ the sequence $\{w_n\}$ has some
subsequence (still denoted by the same) which converges in $(L^2(\Omega))^m$ to
some function $\phi=(\phi_1,\ldots,\phi_m)$ satisfying
\begin{equation}\label{Comp1a}
\int_{B_{R_k}(0)}\sum_{j=1}^m|\phi_j|^2\ dx>\Sigma(M)-\frac{1}{k}.
\end{equation}
Using Cantor diagonalization argument and the fact
$
\sum_{j=1}^m\|w_j^n\|_{L^2}^2=\Sigma(M),\ \forall n,
$
one then concludes that
$(w_n)$ converges (up to a subsequence) strongly to $\phi$ in $(L^2(\mathbb{R}^N))^m $ satisfying
$
\sum_{j=1}^m\|\phi_j\|_{L^2}^2=\Sigma(M).
$
For any $t>0$, we now estimate
\[
\left\vert \mathbb{F}_t(w_k^n, w_j^n)-\mathbb{F}_t(\phi_k, \phi_j)\right\vert
\]
\begin{equation}\label{com1d}
\begin{aligned}
& \leq \frac{1}{t}\iint_{\mathbb{R}^N}W(|x-y|)\left\vert |w_k^n(x)|^p|w_j^n(y)|^p-|\phi_k(x)|^p|\phi_j(y)|^p \right\vert\ dxdy\\
&\leq \frac{1}{t}\iint_{\mathbb{R}^N}W(|x-y|)|w_k^n(x)|^p\left\vert |w_j^n(y)|^p-\phi_j(y)|^p \right\vert\ dxdy\\
&+ \frac{1}{t}\iint_{\mathbb{R}^N}W(|x-y|)|\phi_j(y)|^p\left\vert |w_k^n(x)|^p-|\phi_k(x)|^p \right\vert\ dxdy
\end{aligned}
\end{equation}
Using the Hardy-Littlewood-Sobolev inequality and the
fact
that $\{w_k^n\}_{n\geq 1}$ is bounded in $H^1(\mathbb{R}^N),$ we obtain that
\[
\begin{aligned}
\left\vert \mathbb{F}_t(w_k^n, w_j^n)-\mathbb{F}_t(\phi_k, \phi_j)\right\vert& \leq C \|W\|_{L_w^r}\|w_k^n\|_{L^{\frac{2pr}{2r-1}}}^p\||w_j^n|^p-|\phi_j|^p \|_{L^{\frac{2r}{2r-1}}} \\
&+ C \|W\|_{L_w^r}\|\phi_j\|_{L^{\frac{2pr}{2r-1}}}^p\||w_k^n|^p-|\phi_k|^p \|_{L^{\frac{2r}{2r-1}}}\\
& \leq 2C \||w_k^n|^p-|\phi_k|^p \|_{L^{\frac{2r}{2r-1}}}
\end{aligned}
\]
Next, using the inequality, $||a|^{p-1}a-|b|^{p-1}b|\leq \frac{p}{2}|a-b|(|a|^{p-1}+|b|^{p-1})$, holds for any $a, b\in \mathbb{R}$ and $p\geq 1,$
and
applying Holder's inequality, we obtain that
\[
\begin{aligned}
 \left\vert \mathbb{F}_t(w_k^n, w_j^n)\right.-&\left.\mathbb{F}_t(\phi_k, \phi_j)\right\vert \leq C \left( \int_{\mathbb{R}^N}\left\vert|w_k^n|^p-|\phi_k|^p\right\vert\ dx \right)^{\frac{2r-1}{2r} }\\
& \leq C \left(\int_{\mathbb{R}^{N}}\left(|w_k^n|^{p-1}+|\phi_k|^{p-1} \right)^{\frac{2r}{2r-1}}|w_k^n-\phi_j|^{ \frac{2r}{2r-1}}\ dx \right)^{\frac{2r-1}{2r} }\\
& \leq C \left( \int_{\mathbb{R}^{N}}\left(|w_k^n|^{ \frac{2pr}{2r-1}}+|\phi_k|^{\frac{2pr}{2r-1} } \right)\ dx \right)^{\rho} \|w_k^n-\phi_k \|_{L^{\frac{2pr}{2r-1} }}\\
& \leq C \|w_k^n-\phi_k \|_{L^{\frac{2pr}{2r-1} }}
\end{aligned}
\]
where $\rho= \frac{2r-1}{2r}\left(1-\frac{1}{p}\right).$
Now, using the standard Interpolation inequality and the Sobolev inequality, it follows that
\begin{equation}\label{ConvInq5}
\begin{aligned}
\left\vert \mathbb{F}_t(w_k^n, w_j^n)\right.-\left.\mathbb{F}_t(\phi_k, \phi_j)\right\vert & \leq C\|w_k^n-\phi_k\|_{L^2}^{\lambda^\prime}\|w_k^n-\phi_k\|_{L^{\frac{2N}{2N-2} } }^{1-\lambda^\prime}\\
& \leq C \|w_k^n-\phi_k\|_{L^2}^{\lambda^\prime}
\end{aligned}
\end{equation}
where $\lambda^\prime=(rN-N+2pr-2Npr)/(2pr-Npr).$
The right-hand side of \eqref{ConvInq5} goes to zero since $w_k^n\to \phi_k$ in $L^2.$
Thus, we have that
$\Lim{n\to \infty}\mathbb{F}_t(w_k^n, w_j^n)=\mathbb{F}_t(\phi_k, \phi_j).$
Furthermore, as a consequence of the weak lower semi-continuity of the norm in a Hilbert space, we can assume,
by extracting another
subsequence if necessary, that $w_n\rightharpoonup \phi$ weakly in $Y_m,$ and that
\[
\|\phi\|_{Y_m} =\|(\phi_1, \ldots, \phi_m)\|_{Y_m}\leq \liminf_{n\to \infty}\|w_1^n, \ldots, w_m^n\|_{Y_m}.
\]
Then it follows that
\[
\mathcal{I}(\phi)=\mathcal{I}(\phi_1, \ldots, \phi_m)\leq \lim_{n\to \infty}\mathcal{I}(w_1^n, \ldots, w_m^n)=I_M^{(m)},
\]
and since $w_j^n\to \phi_j$ in $L^2(\mathbb{R}^N),$ we also
have that $\|\phi_j\|_{L^2}^2=\Lim{n\to \infty}\|w_j^n\|_{L^2}^2=M_j$ for $1\leq j\leq m.$
By the definition of the infimum $I_M^{(m)},$ we must have $\mathcal{I}(\phi_1, \ldots, \phi_m)=I_M^{(m)}$ and $u\in \Sigma_M^{(m)}.$
Finally, the facts
$\mathcal{I}(\phi)=\Lim{n\to \infty}\mathcal{I}(w_n)$,
$\mathbb{F}_t(\phi_k, \phi_j)=\Lim{n\to \infty}\mathbb{F}_t(w_k^n, w_j^n),$ and
$\|\phi_j\|_{L^2}=\Lim{n\to \infty}\|w_j^n\|_{L^2}$ together
imply that $\|\phi\|_{Y_m}=\Lim{n\to \infty}\|w_n\|_{Y_m},$ and
from a standard exercise in the elementary Hilbert space theory one
then obtains that $w_n\to \phi$ in $Y_m$ norm.
\end{proof}
We end this section with the following lemma which will be used in the next section to rule out the case of dichotomy.
\begin{lem}\label{revslem}
For any minimizing sequence $\{(u_1^n, \ldots, u_m^n)\}_{n\geq 1}$ of $I^{(m)}_M,$ let $Z^{(m)}$ be
defined by \eqref{lamDef}. Then there exists $T\in [0, M_1]\times \ldots \times [0, M_m]$ such that
\begin{equation}\label{caseCOMP}
Z^{(m)}=T_1+\ldots+T_m\ \ \mathrm{and}\ \ I^{(m)}_T+I^{(m)}_{M-T}\leq I^{(m)}_M.
\end{equation}
\end{lem}
\begin{proof}
The proof is almost same as the proof of Lemma~2.12 of \cite{[SB3]}; we only provide an outline here.
Let $\varepsilon>0$ be arbitrary.
Using the definition of $Z^{(m)}$ and the convergence properties of $Q_k(R),$ there exists $R_0(\varepsilon)$, $k_0(\varepsilon)$ such that
for all $R\geq R_0(\varepsilon)$ and $k\geq k_0(\varepsilon)$, we have that
\begin{equation}\label{epsiInq}
Z^{(m)}-\frac{3\varepsilon}{4}<Q_k(R)\leq Q_k(2R)\leq Z^{(m)}+\frac{3\varepsilon}{4}.
\end{equation}
The inequalities \eqref{epsiInq} together with the definition of $Q_k$ implies that
there exists a sequence of vectors $y_k$ in $\mathbb{R}^N$ such that
\begin{equation}\label{sub113}
\int_{B_{R}(y_k)}\sum_{j=1}^m|u_j^k|^2\ dx> Z^{(m)}-\epsilon,\ \
\int_{B_{2R}(y_k)}\sum_{j=1}^m|u_j^k|^2\ dx < Z^{(m)} +\epsilon.
\end{equation}
Let $\phi\in C_0^{\infty}(\mathbb{R}^N)$ be such that $\phi(x)\equiv 0$ for $|x|\geq 2$ and $\phi(x)\equiv 1$ for $|x|\leq 1,$ and take $\psi\in C^\infty(\mathbb{R}^N)$ such that $\phi^2+\psi^2 \equiv 1$ for $x\in \mathbb{R}^N.$ For any $R>0,$ let $\phi_R$ and $\psi_R$ denote
the rescale functions $\phi_R(x)=\phi(x/R)$ and $\psi_R(x)=\psi(x/R)$ for $x\in \mathbb{R}^N.$ Let us now define
\[
\begin{aligned}
 u_{j,k}^{(1)}=\phi_R(x+y_k)u_j^k,\ \ u_{j,k}^{(2)}=\psi_R(x+y_k)u_j^k,\ 1\leq j\leq m.
\end{aligned}
\]
From Lemma~\ref{NegIM2}, the sequences $\{u_{j,k}^{(1)}\}_{k\geq 1}$ and $\{u_{j,k}^{(2)}\}_{k\geq 1}$, $1\leq j\leq m$
are bounded in $L^2$. Thus, by passing to subsequences, we may assume that there exists $T\in [0, M_1]\times \ldots \times [0, M_m]$
such that $\int_{\mathbb{R}^N}|u_{j,k}^{(1)}|^2\ dx\to T_j$, whence it also follows that $\int_{\mathbb{R}^N}|u_{j,k}^{(2)}|^2\ dx\to M_j-T_j.$
Now we have
\[
T_1+\ldots +T_m=\lim_{k\to \infty}\sum_{j=1}^m\int_{\mathbb{R}^N}|u_{j,k}^{(1)}|^2\ dx=\lim_{k\to \infty}\sum_{j=1}^m\int_{\mathbb{R}^N}\phi_{R}^2|u_j^k|^2\ dx,
\]
where and in what follows we have written the rescaled functions $\phi_R(x+y_k)$
and $\psi_R(x+y_k)$ simply by $\phi_R$ and $\psi_R$ respectively. From \eqref{sub113} it follows that, for any $k\in \mathbb{N}$,
\[
Z^{(m)}-\varepsilon <\sum_{j=1}^m\int_{\mathbb{R}^N}\phi_{R}^2|u_j^k|^2\ dx<Z^{(m)}+\varepsilon.
\]
Then it follows that
\[
\left\vert(T_1+\ldots+T_m)-Z^{(m)}\right\vert<\varepsilon.
\]
Let us write $U_k^{(1)}=(u_{1,k}^{(1)}, \ldots, u_{m,k}^{(1)})$ and $U_k^{(2)}=(u_{1,k}^{(2)}, \ldots, u_{m,k}^{(2)}).$
Then, using a standard argument, one can obtain that
\begin{equation} \label{Eineq}
\mathcal{I}(U_k^{(1)})+\mathcal{I}(U_k^{(2)})\leq \mathcal{I}(u_1^n, \ldots, u_m^k)+C\varepsilon,\ \forall k.
\end{equation}
To prove \eqref{caseCOMP}, since $\{U_k^{(1)}\}_{k\geq 1}$ and $\{U_k^{(2)}\}_{k\geq 1}$
are bounded in $Y_m$, so by passing to a subsequence, we may assume that
$\mathcal{I}(U_k^{(1)}) \to K_1$ and $\mathcal{I}(U_k^{(2)}) \to K_2$, as $k\to \infty.$
Then, since ${\textstyle \lim_{k\to \infty}\mathcal{I}(u_1^k, \ldots, u_m^k)=I_{M}^{(m)}},$ \eqref{Eineq} implies that
$K_1 +K_2 \leq I_{M}^{(m)}+C\epsilon.$ Taking $\varepsilon$ sufficiently small, $R$ sufficiently
large, and making use of results from preceding paragraphs, we can find, for every $a\in \mathbb{N},$
the sequences $\{U_k^{(1,a)}\}$ and $\{U_k^{(2,a)}\}$ in $Y_m$ such that
\[
\begin{aligned}
& \lim_{k\to \infty}\|u_{j, k}^{(1,a)}\|_{L^2}^2=T_j(a),\ \lim_{k\to \infty}\|u_{j, k}^{(2,a)}\|_{L^2}^2=M_j-T_j(a),\ 1\leq j\leq m, \\
& \lim_{k\to \infty}\mathcal{I}\left( U_k^{(i,a)}\right) =\lim_{k\to \infty}\mathcal{I}\left( u_{1,k}^{(i,a)}, \ldots, u_{m,k}^{(i,a)}\right)= K_i(a),\ i=1,2,
\end{aligned}
\]
where $T_j(a)\in [0, M_j]$ and $K_i(a)$ satisfy
\begin{equation}\label{K1K2Limit}
\left\vert \sum_{j=1}^mT_j(a)-Z^{(m)}\right\vert \leq \epsilon\ \ \textrm{and}\ \ K_1(a) +K_2(a) \leq I_M^{(m)}+\frac{1}{a}.
\end{equation}
One can further pass to a subsequence and assume that
$T_j(a)\to T_j\in [0, M_j]$ and $K_i(a)\to K_i.$
Furthermore, after relabeling the sequences
$\{U_k^{(i)}\}$ to be the diagonal subsequences
$ U_k^{(i)}=U_k^{(i, k)}$, $i=1,2,$
we can further assume that
\[
\begin{aligned}
& \lim_{n\to \infty}\|u_{j, k}^{(1)}\|_{L^2}^2=T_j,\ \lim_{n\to \infty}\|u_{j, k}^{(2)}\|_{L^2}^2=M_j-T_j,\ j=1,\ldots,m, \\
& \lim_{n\to \infty}\mathcal{I}( U_k^{(i)})=\lim_{n\to \infty}\mathcal{I}( u_{1,k}^{(i)}, \ldots, u_{m,k}^{(i)})= K_i,\ i=1,2.
\end{aligned}
\]
Now, passing limit as $a \to \infty$ in the first inequality of \eqref{K1K2Limit}, it follows that $Z^{(m)}=T_1+\ldots+T_m.$
In view of the second inequality of \eqref{K1K2Limit}, the proof will be
complete if we are able to deduce that
$K_1 \geq I_T^{(m)}$ and $K_2 \geq I_{M-T}^{(m)}.$
To prove $K_1 \geq I_T^{(m)},$ we consider two cases, namely, $T_j>0$ for all $1\leq j\leq m;$ and
exactly $\widetilde{m}$ of $T_1, \ldots, T_m$ are zero for any $1\leq \widetilde{m}\leq m-1.$
Suppose first that $T_1>0, \ldots, T_m>0.$ Define the numbers
\[
\beta_{j, k}^{(1)}=\frac{\sqrt{T_j}}{\|u_{j,k}^{(1)}\|_{L^2}},\ j=1,\ldots,m.
\]
Then, one has that  $\mathcal{I}(\beta_{1,k}^{(1)}u_{1,k}^{(1)},\ldots,\beta_{m,k}^{(1)}u_{m,k}^{(1)})\geq I_{T}^{(m)}.$
Since $\beta_{j, k}^{(1)}\to 1$ as $k\to \infty,$ it follows that
\[
K_1=\lim_{k\to \infty}\mathcal{I}(\beta_{1,k}^{(1)}u_{1,k}^{(1)},\ldots,\beta_{m,k}^{(1)}u_{m,k}^{(1)})\geq I_{T}^{(m)}.
\]
Now suppose that exactly $\widetilde{m}$ of $T_1, \ldots, T_m$ are zero for any $1\leq \widetilde{m}\leq m-1.$
By relabeling the indices on $T_j$'s, we may assume that
$T_1=0, \ldots, T_{\widetilde{m}}=0$ and $T_{\widetilde{m}+1}>0, \ldots, T_m>0.$
Then, for each $j=1,\ldots,m,$ using the Hardy-Littlehood-Sobolev and
Gagliardo-Nirenberg inequalities, one obtains that
\[
\begin{aligned}
\int_{\mathbb{R}^{N}}(W\star|u_{j,k}^{(1)}|^p)|u_{i,k}^{(1)}|^p\ dx& \leq C \|W\|_{L_w^r}\|u_{j,k}^{(1)} \|_{L^{ \frac{2pr} {2r-1}} }^p\| u_{i,k}^{(1)} \|_{L^{\frac{2pr}{2r-1}}}^p \\
& \leq C \|\nabla u_{j,k}^{(1)} \|_{L^2}^{\mu p}\|u_{j,k}^{(1)} \|_{L^2}^{(1-\mu) p}\|\nabla u_{i,k}^{(1)} \|_{L^2}^{\mu p}\| u_{i,k}^{(1)} \|_{L^2}^{(1-\mu) p}\\
& \leq C \| u_{i,k}^{(1)} \|_{L^2}^{(1-\mu) p}\to 0,\ 1\leq i\leq \widetilde{m},
\end{aligned}
\]
as $k\to \infty,$
where $\mu=N(pr-2r+1)/2rp.$
In consequence, we obtain that
\[
\begin{aligned}
K_1& =\lim_{k\to \infty}\mathcal{I}(U_k^{(1)})=\lim_{k\to \infty}\mathcal{I}(u_{1,k}^{(1)}, \ldots, u_{m,k}^{(1)}) \\
& =\lim_{k\to \infty}\left(\sum_{j=1}^m\|\nabla u_{j,k}^{(1)}\|_{L^2}^2-\frac{1}{2p}\sum_{i,j=\widetilde{m}+1}^m\int_{\mathbb{R}^{N}}(W\star |u_{i,k}^{(1)}|^p|u_{j,k}^{(1)}|^p\ dx \right)\\
& \geq \liminf_{k\to \infty}\left(\sum_{j=\widetilde{m}+1}^m\|\nabla u_{j,k}^{(1)}\|_{L^2}^2-\frac{1}{2p}\sum_{i,j=\widetilde{m}+1}^m\int_{\mathbb{R}^{N}}(W\star |u_{i,k}^{(1)}|^p|u_{j,k}^{(1)}|^p\ dx \right)\\
& \geq I_{0,\ldots,0, T_{\widetilde{m}+1 }, \ldots, T_m}^{(m)}=I_T^m.
\end{aligned}
\]
To prove that $K_2 \geq I_{M-T}^{(m)},$ one can go through the same
argument as in the proof of $K_1 \geq I_{T}^{(m)}$ by
treating $M_1-T_1, \ldots ,M_m-T_m$ as $T_1, \ldots, T_m,$ respectively.
\end{proof}

\section{The problem with two constraints}\label{section2}
In this section, we follow the method developed in \cite{[AB11]} to rule out the possible dichotomy of the minimizing sequences.
For this purpose, we require to prove the strict subadditivity inequality
for the function $I_M^{(2)}$.

\smallskip

In the sequel we shall use the following notation:
\[
\langle E\rangle(f_1, f_2, \ldots, f_m)=E(f_1)+E(f_2)+\ldots + E(f_m),
\]
where the functional $E$ is as defined in Lemma~\ref{tech1}.
The strict subadditivity under two constraints takes the following form:
\begin{lem}\label{2MainSlemIa}
Let $\mathbb{R}_{\geq 0}=[0, \infty).$ For any $M, T\in \mathbb{R}_{\geq 0}^2$ satisfying
$M, T\neq \{\mathbf{0}\}$ and $S=M+T\in \mathbb{R}_{+}^2,$ one has
\begin{equation}\label{2MainSubadd}
I^{(2)}_S<I^{(2)}_M+I^{(2)}_T.
\end{equation}
\end{lem}
To prove Lemma~\ref{2MainSlemIa}, we use ideas from \cite{[SB3], SBC}.
Since $M_1+T_1>0$, the following cases arise: $M_1>0$ and $T_1>0$; $M_1=0$ and $T_1>0$; or
$M_1>0$ and $T_1=0.$ The third case can be reduced to the second
case by switching $M_1$ and $T_1$ and so we do not consider it.
In the first case, since $M_2+T_2>0,$ the following cases may arise:
\begin{itemize}
\item[$(a_1)$] $M_1>0, T_1>0, M_2>0,$ and $T_2>0,$
\item[$(a_2)$] $M_1>0, T_1>0, M_2=0,$ and $T_2>0,$
\item[$(a_3)$] $M_1>0, T_1>0, M_2>0,$ and $T_2=0.$
\end{itemize}
In the second case, since $M_1+M_2>0$, $T_1+T_2>0$, and $M_2+T_2>0,$ the following cases may arise:
\begin{itemize}
\item[$(b_1)$] $M_1=0$, $T_1>0$, $M_2>0,$ and $T_2>0,$
\item[$(b_2)$] $M_1=0$, $T_1>0$, $M_2>0,$ and $T_2=0.$
\end{itemize}
In order to prove Lemma~\ref{2MainSlemIa}, it suffices to consider the cases $(a_1)$, $(b_1)$, and $(b_2)$. All other cases can be reduced to one of these cases by switching roles of $M_j$'s and $T_j$'s.
We consider these three cases in the next three lemmas.

\smallskip

The first lemma concerns the case $(a_1).$
\begin{lem}
Let $M=(M_1,M_2)\in \mathbb{R}_{+}^2$ and
 $T=(T_1, T_2)\in \mathbb{R}_{+}^2$. Then one has
 \begin{equation}\label{2casea1}
 I^{(2)}_{M+T}<I^{(2)}_M+I^{(2)}_T.
 \end{equation}
\end{lem}
\begin{proof}
We follow the ideas from \cite{[SB3], SBC}.
Let $\{(u_1^n, u_2^n)\}$ and $\{(v_1^n, v_2^n)\}$ be any sequences in $Y_2$ satisfying
\[
\begin{aligned}
& \lim_{n\to \infty}\|u_j^n\|_{L^2}^2=M_j,\ \ \lim_{n\to \infty}\|v_j^n\|_{L^2}^2=T_j,\ j=1,2,\\
& \lim_{n\to \infty}\mathcal{I}(u_1^n, u_2^n)=I^{(2)}_M,\ \ \lim_{n\to \infty}\mathcal{I}(v_1^n, v_2^n)=I^{(2)}_T.
\end{aligned}
\]
By passing to a subsequence if necessary, we may assume that the following values exist.
\[
\begin{aligned}
& A_1=\frac{1}{M_1}\lim_{n\to \infty}\left(E(u_1^n)-\mathbb{F}_p(u_1^n, u_2^n) \right),\ B_1=\frac{1}{M_2}\lim_{n\to \infty}E(u_2^n),\\
& A_2=\frac{1}{T_1}\lim_{n\to \infty}\left(E(v_1^n)-\mathbb{F}_p(v_1^n, v_2^n) \right),\ B_2=\frac{1}{T_2}\lim_{n\to \infty}E(v_2^n).
\end{aligned}
\]
To prove \eqref{2casea1}, we consider three cases: $A_1<A_2$; $A_1>A_2$; and $A_1=A_2$. Assume first that $A_1<A_2.$
Without loss of generality, we may assume that $u_j^n$ and $v_j^n$ are non-negative and
by density argument, we may also assume that $u_j^n$ and $v_j^n$ have compact supports.
Let $\widetilde{v}_2^n=v_2^n(\cdot - x_n\rho)$, where $\rho$ is some unit vector in $\mathbb{R}^N$ and $x_n$ is chosen such that $x_n\to 0$ as $n\to \infty$, and
$\widetilde{v}_2^n$ and $u_2^n$ have disjoint supports.
Define $(f_1^n, f_2^n)$ as follows:
$
f_1^n=\ell^{1/2}u_1^n$ and $f_2^n=u_2^n+\widetilde{v}_2^n,
$
where $\ell=(M_1+T_1)/M_1.$
Then we have that
\begin{equation}\label{Pfcaseq1}
\begin{aligned}
I^{(2)}_{M+T}& \leq \lim_{n\to \infty}\mathcal{I}(f_1^n, f_2^n)\\
& =\lim_{n\to \infty}\left(E(f_1^n)+\langle E\rangle(u_2^n,\widetilde{v}_2^n)-\mathbb{F}_p(f_1^n, f_2^n) \right)\\
&\leq \lim_{n\to \infty}\left(E(f_1^n)+\langle E\rangle(u_2^n,\widetilde{v}_2^n)-\mathbb{F}_p(f_1^n, u_2^n) \right).
\end{aligned}
\end{equation}
Since $\ell>1$ and $p\geq 2,$ we have that $\ell^{p/2}\geq \ell.$ Then it follows that
\begin{equation}\label{Pfcaseq2}
\begin{aligned}
 \mathbb{F}_{2p}(\ell^{1/2}f,\ell^{1/2}f)& =\ell^p \mathbb{F}_{2p}(f,f)\geq \ell \mathbb{F}_{2p}(f,f),\\
 \mathbb{F}_p(\ell^{1/2}f,g)& =\ell^{p/2} \mathbb{F}_p(f,g)\geq \ell \mathbb{F}_p(f,g).
\end{aligned}
\end{equation}
Making use of these observations, we obtain that
\begin{equation}\label{Pfcaseq3}
\begin{aligned}
E(f_1^n)-\mathbb{F}_p(f_1^n, u_2^n)& =\ell\|\nabla u_1^n\|_{L^2}^2-\mathbb{F}_{2p}(\ell^{1/2}u_1^n,\ell^{1/2}u_1^n)-\mathbb{F}_p(\ell^{1/2}u_1^n,u_2^n)\\
&\leq \ell\left( \|\nabla u_1^n\|_{L^2}^2-\mathbb{F}_{2p}(u_1^n,u_1^n)-\mathbb{F}_p(u_1^n,u_2^n)\right)\\
& = E(u_1^n)-\mathbb{F}_p(u_1^n,u_2^n)+\frac{T_1}{M_1}\left( E(u_1^n)-\mathbb{F}_p(u_1^n,u_2^n)\right)
\end{aligned}
\end{equation}
Using \eqref{Pfcaseq1}, \eqref{Pfcaseq3}, and the fact $A_1<A_2$, it follows that
\[
\begin{aligned}
I^{(2)}_{M+T}&\leq I^{(2)}_{M}+\lim_{n\to \infty}E(\widetilde{v}_2^n)+\frac{T_1}{M_1}(A_1M_1)\\
& < I^{(2)}_{M}+\lim_{n\to \infty}E(\widetilde{v}_2^n)+T_1A_2\\
& = I^{(2)}_{M}+\lim_{n\to \infty}\left(\langle E\rangle(v_1^n, \widetilde{v}_2^n)-\mathbb{F}_p(v_1^n, \widetilde{v}_2^n) \right)\\
& = I^{(2)}_{M}+\lim_{n\to \infty}\mathcal{I}(v_1^n, \widetilde{v}_2^n)
=I^{(2)}_{M}+I^{(2)}_{T}.
\end{aligned}
\]
The proof in the case $A_1>A_2$ goes through unchanged after swapping the indices and so we do not repeat here.
Next suppose that $A_1=A_2$. We consider two subcases: $B_1\leq B_2$ and $B_1\geq B_2.$
Suppose first that $A_1=A_2$ and $B_1\leq B_2.$ Let $\ell$ be defined as above and $s=(M_2+T_2)/M_2.$
Then we have that
\begin{equation}\label{Pfcaseq4}
\begin{aligned}
I^{(2)}_{M+T}& \leq \lim_{n\to\infty}\mathcal{I}(\ell^{1/2}u_1^n, s^{1/2}u_2^n)\\
& =\lim_{n\to\infty}\left(E(\ell^{1/2}u_1^n)+E(s^{1/2}u_2^n)-\mathbb{F}_p(\ell^{1/2}u_1^n,s^{1/2}u_2^n) \right)
\end{aligned}
\end{equation}
Since $\ell>1$, $s>1$, and $p\geq 2,$ we have that
$\ell^{p/2}\geq \ell$ and $s^{p/2}\geq s>1.$ It follows that
\[
\begin{aligned}
 \mathbb{F}_{2p}(\ell^{1/2}u_1^n,\ell^{1/2}u_1^n)&=\ell^p\mathbb{F}_{2p}(u_1^n,u_1^n)\geq \ell \mathbb{F}_{2p}(u_1^n,u_1^n) ,\\
 \mathbb{F}_p(\ell^{1/2}u_1^n,s^{1/2}u_2^n)&=\ell^{p/2}s^{p/2}\mathbb{F}_p(u_1^n,u_2^n)\geq \ell \mathbb{F}_p(u_1^n,u_2^n).
\end{aligned}
\]
Using these observations, a similar argument as in \eqref{Pfcaseq3} yields
\begin{equation}\label{Pfcaseq5}
\begin{aligned}
E(\ell^{1/2}u_1^n)-\mathbb{F}_p&(\ell^{1/2}u_1^n,s^{1/2}u_2^n)\leq \ell\left( E(u_1^n)-\mathbb{F}_p(u_1^n, u_2^n)\right)\\
& = E(u_1^n)-\mathbb{F}_p(u_1^n, u_2^n)+\frac{T_1}{M_1}\left(E(u_1^n)-\mathbb{F}_p(u_1^n, u_2^n) \right)
\end{aligned}
\end{equation}
Using Lemma~\ref{tech1}, there exists $\delta>0$ such that for sufficiently large $n,$
\begin{equation}\label{Pfcaseq6}
\begin{aligned}
E(s^{1/2}u_2^n)\leq sE(u_2^n)-\delta=E(u_2^n)+\frac{T_2}{M_2}E(u_2^n)-\delta.
\end{aligned}
\end{equation}
Inserting \eqref{Pfcaseq5} and \eqref{Pfcaseq6} into \eqref{Pfcaseq4} and using the assumptions
$A_1=A_2$ and $B_1\leq B_2$, we obtain that
\[
\begin{aligned}
I^{(2)}_{M+T}& \leq \lim_{n\to \infty}\left(\mathcal{I}(u_1^n, u_2^n)+\frac{T_2}{M_2}E(u_2^n)+\frac{T_1}{M_1}\left(E(u_1^n)-\mathbb{F}_p(u_1^n, u_2^n) \right)\right)-\delta\\
& =I^{(2)}_{M}+\frac{T_2}{M_2}(M_2B_1)+\frac{T_1}{M_1}(T_1A_1)-\delta\\
&\leq I^{(2)}_{M}+T_2B_2+T_1A_2-\delta\\
&=I^{(2)}_{M}+\lim_{n\to \infty}\mathcal{I}(v_1^n, v_2^n)-\delta=I^{(2)}_{M}+I^{(2)}_{T}-\delta,
\end{aligned}
\]
which gives the desired strict inequality. The proof in the case $A_1=A_2$ and $B_1\geq B_2$ follows a similar argument and we do not repeat here.
\end{proof}
The following lemma establishes \eqref{2MainSubadd} in the case $(b_1).$
\begin{lem}
For any $T\in \mathbb{R}_{+}^2$ and $M=(0, M_2)$ with $M_2>0,$ one has
\[
I^{(2)}_{M+T}<I^{(2)}_M+I^{(2)}_T.
\]
\end{lem}
\begin{proof}
Let $\{(0, u_2^n)\}$ and $\{(v_1^n, v_2^n)\}$ be any sequences in $Y_2$ satisfying
\[
\begin{aligned}
&  \lim_{n\to \infty}\|u_2^n\|_{L^2}^2=M_2,\ \lim_{n\to \infty}\|v_j^n\|_{L^2}^2=T_j, \\
& \lim_{n\to \infty}\mathcal{I}(0, u_2^n)=I^{(2)}_M, \ \ \mathrm{and}\ \lim_{n\to \infty}\mathcal{I}(v_1^n, v_2^n)=I^{(2)}_T.
\end{aligned}
\]
As in the previous lemma, after passing to a subsequence if necessary, we consider the following values
\[
\begin{aligned}
D_1=\frac{1}{M_2}\lim_{n\to \infty}\left(E(u_2^n)-\mathbb{F}_p(v_1^n, u_2^n) \right),\
 D_2=\frac{1}{T_2}\lim_{n\to \infty}\left(E(v_2^n)-\mathbb{F}_p(v_1^n, v_2^n) \right).
\end{aligned}
\]
We consider three cases: $D_1<D_2$, $D_1>D_2$, and $D_1=D_2.$ Assume first that $D_1<D_2.$
Let $s=(M_2+T_2)/M_2.$ Since $s>1$ and $p\geq 2,$ it follows that
\begin{equation}\label{2Lcaseb1a}
\begin{aligned}
I^{(2)}_{M+T}& \leq \lim_{n\to \infty}\mathcal{I}(v_1^n, s^{1/2}u_2^n)\\
& =\lim_{n\to \infty}\left(E(v_1^n)+s\|\nabla u_2^n\|_{L^2}^2-s^p\mathbb{F}_{2p}(u_2^n, u_2^n)-s^{p/2}\mathbb{F}_p(v_1^n, u_2^n) \right)\\
& \leq \lim_{n\to \infty}\left(E(v_1^n)+sE(u_2^n)-s\mathbb{F}_p(v_1^n, u_2^n) \right).
\end{aligned}
\end{equation}
Since $E(f)-\mathbb{F}_p(f,g)\leq \mathcal{I}(0,f)$ and $D_1<D_2$, it follows from \eqref{2Lcaseb1a} that
\begin{equation}\label{2Lcaseb1a2}
\begin{aligned}
I^{(2)}_{M+T}& \leq I^{(2)}_M+\lim_{n\to \infty}\left(E(v_1^n)+\frac{T_2}{M_2}\left(E(u_2^n)-\mathbb{F}_p(v_1^n, u_2^n) \right) \right)\\
& = I^{(2)}_M+\lim_{n\to \infty}E(v_1^n)+\frac{T_2}{M_2}(M_2D_1)\\
& < I^{(2)}_M+\lim_{n\to \infty}E(v_1^n)+T_2D_2\\
&=I^{(2)}_M+\lim_{n\to \infty}\mathcal{I}(v_1^n, v_2^n)
=I^{(2)}_M+I^{(2)}_T,
\end{aligned}
\end{equation}
which is the desired strict inequality. The proof in the case $D_1>D_2$ follows the same steps and we omit the details.
Now consider the case that $D_1=D_2.$ Let $f_2^n=s^{1/2}u_2^n$, where $s$ is defined as above. Then, using Lemma~\ref{tech1}, there exists
a number $\delta>0$ such that for sufficiently large $n,$
\begin{equation}\label{2Lcaseb1a3}
E(f_2^n)=E(s^{1/2}u_2^n)\leq sE(u_2^n)-\delta.
\end{equation}
Since $s>1$ and $p\geq 2$, we have that $s^{p/2}\geq s.$ Then it is easy to see that $\mathbb{F}_p(f,s^{1/2}g)\geq s\mathbb{F}_p(f,g).$
Using this observation and \eqref{2Lcaseb1a3}, we obtain that
\begin{equation}\label{2Lcaseb1a4}
\begin{aligned}
I^{(2)}_{M+T}&\leq \lim_{n\to \infty}\mathcal{I}(v_1^n, s^{1/2}u_2^n)\\
& = \lim_{n\to \infty}\left(E(v_1^n)+E(s^{1/2}u_2^n)-\mathbb{F}_p(v_1^n, s^{1/2}u_2^n) \right)\\
& \leq \lim_{n\to \infty}\left(E(v_1^n)+sE(u_2^n)-s\mathbb{F}_p(v_1^n, u_2^n) \right)-\delta.
\end{aligned}
\end{equation}
Once we have obtained \eqref{2Lcaseb1a4}, the desired strict inequality follows using the same lines as in \eqref{2Lcaseb1a2}.
\end{proof}
To complete the proof of Lemma~\ref{2MainSlemIa}, it only remains to establish \eqref{2MainSubadd} in the case $(b_2).$
This will be done in the next lemma.
\begin{lem}
For any $M\in \{0\}\times \mathbb{R}_{+}$ and $T\in \mathbb{R}_{+}\times \{0\}$, one has
\[
I^{(2)}_{M+T}<I^{(2)}_M+I^{(2)}_T.
\]
\end{lem}
\begin{proof}
Using Lemma~\ref{tech3}, let $\phi_{M_2}>0$ and $\phi_{T_1}>0$ be such that
\[
\begin{aligned}
& E(\phi_{M_2})=\inf\left\{E(f):f\in H^1(\mathbb{R}^N)\ \mathrm{and}\ \int_{\mathbb{R}^N}|f|^2\ dx=M_2 \right\},\\
& E(\phi_{T_1})=\inf\left\{E(f):f\in H^1(\mathbb{R}^N)\ \mathrm{and}\ \int_{\mathbb{R}^N}|f|^2\ dx=T_1 \right\}.
\end{aligned}
\]
Then it is obvious that $\mathbb{F}_p(\phi_{M_2},\phi_{T_1})>0.$ Thus, it follows that
\[
\begin{aligned}
I^{(2)}_{M+T}&\leq \mathcal{I}(\phi_{M_2},\phi_{T_1})=E(\phi_{M_2})+E(\phi_{T_1})-\mathbb{F}_p(\phi_{M_2},\phi_{T_1})\\
&=I^{(2)}_M+I^{(2)}_T-\mathbb{F}_p(\phi_{M_2},\phi_{T_1})<I^{(2)}_M+I^{(2)}_T,
\end{aligned}
\]
which is the desired strict inequality.
\end{proof}
We are now able to rule out the case $0<Z^{(2)}<M_1+M_2.$

\begin{lem}\label{dicMainm2}
Suppose that $\{u_1^n,u_2^n\}_{n\geq 1}\subset Y_2$ be any minimizing
sequence of $I^{(2)}_M$ and $Z^{(2)}$ be defined by \eqref{lamDef} with $m=2.$
Then, one has
\[
Z^{(2)}=M_1+M_2.
\]
\end{lem}
\begin{proof}
Since the case $Z^{(2)}=0$ has been ruled out, we show that $Z^{(2)}\not\in (0, M_1+M_2).$
Suppose that $Z^{(2)} \in (0, M_1+M_2)$ holds. Let $T$ be defined as in Lemma~\ref{revslem}
and define $S=(S_1, S_2)$ by $S_j=M_j-T_j$, $j=1,2.$ Then, we have that
$S+T\in \mathbb{R}_{+}^2.$ We also have $T_1+T_2=Z^{(2)}>0$ and
\[
S_1+S_2=M_1+M_2-(T_1+T_2)=M_1+M_2-Z^{(2)}>0.
\]
Applying Lemma~\ref{2MainSlemIa}, we then have
\[
I_T^{(2)}+I_S^{(2)}>I_{S+T}^{(2)}.
\]
This is same as $I_T^{(2)}+I_{M-T}^{(2)}>I_M,$ contradicting the result of Lemma~\ref{revslem}. This proves that
$Z^{(2)}\not\in (0, M_1+M_2)$ and we must have $Z^{(2)}= M_1+M_2.$
\end{proof}

\begin{lem}\label{2MainlemEx}
For every $M\in \mathbb{R}_{+}^2$, the set $\Lambda^{(2)}(M)$ is nonempty. Moreover, the following statements hold.

\smallskip

(i) For every $(\phi_1, \phi_2)\in \Lambda^{(2)}(M)$, there exists $\lambda_1$ and $\lambda_2$ such that
\begin{equation}\label{SWdef2a}
(\psi_1(x,t), \psi_2(x,t))=(e^{-i\lambda_1 t} \phi_1(x),e^{-i\lambda_2 t} \phi_2(x))
\end{equation}
is a standing-wave solution of \eqref{NChoM2} with $m=2.$

\smallskip

(ii) The Lagrange multipliers $\lambda_1$ and $\lambda_2$ satisfy $\lambda_j>0.$

\smallskip

(iii) For every $(\phi_1,\phi_2)\in \Lambda^{(2)}(M)$ there exists $\theta_j\in \mathbb{R}$ and real-valued functions
$\phi_{M_1}$ and $\phi_{M_2}$ such that
\[
\phi_{M_j}(x)>0\ \ \mathrm{and}\ \ \phi_j(x)=e^{i\theta_j}\phi_{M_j}(x),\ x\in \mathbb{R}^N.
\]
\end{lem}
\begin{proof}
Let $(\phi_1, \phi_2)\in \Lambda^{(2)}(M).$ Then
the Lagrange multiplier principle implies that each function $(\phi_1, \phi_2)$
satisfies Euler-Lagrange equations
\begin{equation}\label{LagEqns}
-\Delta \phi_j+\lambda_j \phi_j=\sum_{k=1}^2(W\star |\phi_k|^p)|\phi_j|^p\phi_j,\ 1\leq j\leq 2,
\end{equation}
where $\lambda_1$ and $\lambda_2$ are Lagrange multipliers. Consequently the function $(\psi_1, \psi_2)$ defined
by \eqref{SWdef2a} is a standing wave for \eqref{NChoM2} with $m=2.$
Multiplying the first equation by $\overline{\phi_1}$ and the section equation by $\overline{\phi_2}$, and integrating by parts, we get
\begin{equation}\label{LagEqnspo}
\begin{aligned}
-\lambda_j\|\phi_j\|_{L^2}^2&=\|\nabla \phi_j\|_{L^2}^2-\sum_{k,j=1}^2\int_{\mathbb{R}^N\times \mathbb{R}^N}W(x-y)Q(\phi_k,\phi_j)\ dxdy\\
& =\|\nabla \phi_j\|_{L^2}^2-2p\left(\mathbb{F}_{2p}(\phi_1,\phi_1)+\mathbb{F}_{2p}(\phi_2,\phi_2)+\mathbb{F}_{p}(\phi_1,\phi_2)\right).
\end{aligned}
\end{equation}
Applying Lemma~\ref{VarE3b} with $(u_1^n, u_2^n)=(\phi_1, \phi_2)$, it follows that there exists $\delta_j>0$ such that
\[
\|\nabla \phi_j\|_{L^2}^2-2\mathbb{F}_{2p}(\phi_1,\phi_1)-2\mathbb{F}_{2p}(\phi_2,\phi_2)-2\mathbb{F}_{p}(\phi_1,\phi_2)< 0.
\]
Since $2p>2$, it follows that the right-hand side of \eqref{LagEqnspo} is negative.
Then it follows that $\lambda_j$ must be positive.

\smallskip

Next, let $(\phi_1, \phi_2)\in \Lambda^{(2)}(M)$ be a complex-valued minimizer of $I_{M_1, M_2}^{(2)}.$ Using the fact that
\[
u\in H^1(\mathbb{R}^N) \Rightarrow |u|\in H^1(\mathbb{R}^N),\ \|\nabla |u|\|_{L^2}\leq \|\nabla u\|_{L^2},
\]
it follows that $(|\phi_1|, |\phi_2|)\in \Lambda^{(2)}(M)$ as well.
By the strong maximum principle, we infer that
\[
|\phi_1|>0 \ \mathrm{and}\ |\phi_2|>0.
\]
We have that
\begin{equation}\label{LagEnspox}
\mathcal{I}(|\phi_1|,|\phi_2|)-\mathcal{I}(\phi_1, \phi_2) =\frac{1}{2}\sum_{j=1}^2\|\nabla |\phi_j|\|_{L^2}^2 -\frac{1}{2}\sum_{j=1}^2 \|\nabla \phi_j\|_{L^2}^2.
\end{equation}
Since both $(\phi_1, \phi_2)$ and $(|\phi_1|, |\phi_2|)$ belong to $\Lambda^{(2)}(M),$ the only possibility \eqref{LagEnspox} can happen is that
\begin{equation}\label{ethePf}
\int_{\mathbb{R}^N}|\nabla \phi_j|^2\ dx=\int_{\mathbb{R}^N}|\nabla |\phi_j||^2\ dx,\ j=1,2.
\end{equation}
Once we have \eqref{ethePf}, a number of techniques are
available to prove item (iii) of Lemma~\ref{2MainlemEx} (see for example, Theorem~5 of \cite{Bell}).
\end{proof}

\section{The problem with three constraints}\label{section3}

In this section we prove the strict subadditivity inequality for $I^{(3)}_M$
and rule out the possible dichotomy of the minimizing sequences.
Throughout this section we shall use the following notation:
\[
\begin{aligned}
& \mathcal{Q}(f,g,h)=E(g)-\mathbb{F}_p(f,g)-\mathbb{F}_p(g,h),\\
& \mathcal{D}(f,g,h)=\mathcal{Q}(f,g,h)-\mathbb{F}_p(f,h),
\end{aligned}
\]
where the functional $E$ is as defined in Lemma~\ref{tech1}.
With these definitions, we can write
\begin{equation}\label{tech3Suba}
\begin{aligned}
\mathcal{I}(f,g,h)&=\mathcal{Q}(f,g,h)+\langle E\rangle(f,h)-\mathbb{F}_p(f,h)\\
& = \mathcal{D}(f,g,h)+\langle E\rangle(f,h)\\
& = \mathcal{Q}(g,f,h)+\langle E\rangle(g,h)-\mathbb{F}_p(g,h).
\end{aligned}
\end{equation}
 The strict subadditivity
condition for the function $I^{(3)}_M$ takes the following form
\begin{lem}\label{MainSlemI}
Let $\mathbb{R}_{\geq 0}=[0, \infty).$ For any $M, T\in \mathbb{R}_{\geq 0}^3$ satisfying
$M, T\neq \{\mathbf{0}\}$ and $S=M+T\in \mathbb{R}_{+}^3,$ one has
\begin{equation}\label{MainSubadd}
I^{(3)}_S<I^{(3)}_M+I^{(3)}_T.
\end{equation}
\end{lem}

\textbf{Proof of Lemma~\ref{MainSlemI}.}
We use ideas from \cite{SBC, [SB321], Lope}.
Since
$M_1+T_1>0,$
we have the following possibilities: $M_1>0$ and $T_1>0;$ or $M_1=0$ and $T_1>0;$ or
$M_1>0$ and $T_1=0.$ The third case $M_1>0$ and $T_1=0$ can be
reduced to the second case by switching $M_1$ and $T_1$ and so we do not consider it.

\smallskip

In the case when $M_1>0$ and $T_1>0,$ the following situations may arise:
\begin{itemize}
\item[$(A_1)$] $M_1>0,\ T_1>0,\ T_2>0,\ M_2=0,\ T_3=0,\ M_3>0,$
\item[$(A_2)$] $M_1>0,\ T_1>0,\ T_2>0,\ M_2=0,\ T_3>0,\ M_3=0,$
\item[$(A_3)$] $M_1>0,\ T_1>0,\ T_2>0,\ M_2=0,\ T_3>0,\ M_3>0,$
\item[$(A_4)$] $M_1>0,\ T_1>0,\ T_2=0,\ M_2>0,\ T_3=0,\ M_3>0,$
\item[$(A_5)$] $M_1>0,\ T_1>0,\ T_2=0,\ M_2>0,\ T_3>0,\ M_3=0,$
\item[$(A_6)$] $M_1>0,\ T_1>0,\ T_2=0,\ M_2>0,\ T_3>0,\ M_3>0,$
\item[$(A_7)$] $M_1>0,\ T_1>0,\ T_2>0,\ M_2>0,\ T_3=0,\ M_3>0,$
\item[$(A_8)$] $M_1>0,\ T_1>0,\ T_2>0,\ M_2>0,\ T_3>0,\ M_3=0,$
\item[$(A_9)$] $M_1>0,\ T_1>0,\ T_2>0,\ M_2>0,\ T_3>0,\ M_3>0.$
\end{itemize}
Similarly, in the second case, i.e., when $T_1>0$ and $M_1=0,$ one has to consider the following cases:
\begin{itemize}
\item[$(B_1)$] $T_1>0,\ M_1=0,\ T_2=0,\ M_2>0,\ T_3=0,\ M_3>0,$
\item[$(B_2)$] $T_1>0,\ M_1=0,\ T_2=0,\ M_2>0,\ T_3>0,\ M_3=0,$
\item[$(B_3)$] $T_1>0,\ M_1=0,\ T_2=0,\ M_2>0,\ T_3>0,\ M_3>0,$
\item[$(B_4)$] $T_1>0,\ M_1=0,\ T_2>0,\ M_2=0,\ T_3=0,\ M_3>0,$
\item[$(B_5)$] $T_1>0,\ M_1=0,\ T_2>0,\ M_2=0,\ T_3>0,\ M_3>0,$
\item[$(B_6)$] $T_1>0,\ M_1=0,\ T_2>0,\ M_2>0,\ T_3=0,\ M_3>0,$
\item[$(B_7)$] $T_1>0,\ M_1=0,\ T_2>0,\ M_2>0,\ T_3>0,\ M_3=0,$
\item[$(B_8)$] $T_1>0,\ M_1=0,\ T_2>0,\ M_2>0,\ T_3>0,\ M_3>0.$
\end{itemize}
To prove the lemma, it suffices to consider the cases $(A_9)$, $(A_3)$, $(B_3)$, $(B_5)$, and $(B_2)$; since otherwise we can
switch the role of the parameters and reduce to one of these cases. We consider each of these cases separately in the next five lemmas.

\smallskip

Before we begin, we make the following observation.
For any $\ell>1,$ define $u_\ell=\ell^{1/2}u_2$ and let $U=(u_1, u_\ell, u_3).$ Then we have that
\begin{equation}\label{tech3Subb}
\begin{aligned}
\mathcal{Q}(U)&=\frac{\ell}{2}\|\nabla u_2\|_{L^2}^2-\ell^p\mathbb{F}_{2p}(u_2,u_2)-\ell^{p/2}\mathbb{F}_p(u_1,u_2)-\ell^{p/2}\mathbb{F}_p(u_2,u_3)\\
&\leq \ell \left( \|\nabla u_2\|_{L^2}^2-\mathbb{F}_{2p}(u_2,u_2)- \mathbb{F}_p(u_1,u_2)-\mathbb{F}_p(u_2,u_3)\right)\\
& = \ell \left( E(u_2)-\mathbb{F}_p(u_1,u_2)-\mathbb{F}_p(u_2,u_3)\right)=\ell \mathcal{Q}(u_1, u_2, u_3).
\end{aligned}
\end{equation}

\smallskip

The following lemma establishes \eqref{MainSubadd} in the case $(A_9).$
\begin{lem}
For any $M, T\in \mathbb{R}_{+}^3,$ one has $I^{(3)}_{M+T}<I^{(3)}_M+I^{(3)}_T.$
\end{lem}
\begin{proof}
For every $M, T\in \mathbb{R}_{+}^3,$ let $\{(u_1^n, u_2^n, u_3^n)\}_{n\geq 1}$ and $\{(v_1^n, v_2^n, v_3^n)\}_{n\geq 1}$
be minimizing sequences
for $I^{(3)}_M$ and $I^{(3)}_T,$ respectively.
Without loss of generality, we may assume that $u_j^n$'s and $v_j^n$'s are real-valued, have compact supports, and
\[
\|u_j^n\|_{L^2}^2=M_j\ \ \ \mathrm{and}\ \ \|v_j^n\|_{L^2}^2=T_j,\ \forall n,\ j=1,2,3.
\]
Define the pair of numbers $(L_1, L_2)\in \mathbb{R}^2$ as follows
\[
L_1=T_1 \lim_{n\to \infty}\mathcal{Q}(u_2^n, u_1^n, u_3^n)\ \ \mathrm{and}\ \ L_2=M_1\lim_{n\to \infty}\mathcal{Q}(v_2^n, v_1^n, v_3^n).
\]
Then, the following situations may occur: $L_1<L_2$ or $L_1>L_2,$ or $L_1=L_2.$ Assume first that $L_1<L_2.$
Define
\begin{equation}\label{ghDef}
\begin{aligned}
& \widetilde{v}_{2}^{n}(\cdot)=v_2^n(\cdot+x_n\rho),\ f_2^n=u_2^n+\widetilde{v}_{2}^{n},\\
& \widetilde{v}_{3}^{n}(\cdot)=v_3^n(\cdot+x_n\rho),\ f_3^n=u_3^n+\widetilde{v}_{3}^{n},
\end{aligned}
\end{equation}
where $\rho$ is a unit vector in $\mathbb{R}^N$ and $x_n$ is such that $x_n\to 0$ as $n\to \infty$ ;
$\widetilde{v}_{2}^{n}$ and $u_2^n$ have disjoint supports; and $\widetilde{v}_{3}^{n}$ and $u_3^n$ have disjoint supports.
Let ${\textstyle \ell=1+\frac{T_1}{M_1}}$ and take the function
$
f_1^n=\ell^{1/2}u_1^n.
$
Then $(f_1^n, f_2^n, f_3^n)\in \Sigma_{M+T}^{(3)}$ and we have
\begin{equation}\label{caseiaMI}
\begin{aligned}
I^{(3)}_{M+T}& \leq \lim_{n\to \infty}\mathcal{I}(f_1^n, f_2^n, f_3^n)\\
& =\lim_{n\to \infty}\left(\langle E\rangle(f_1^n, f_2^n, f_3^n)-\mathbb{F}_p(f_1^n, f_2^n)-\mathbb{F}_p(f_2^n, f_3^n)-\mathbb{F}_p(f_1^n, f_3^n) \right)\\
& \leq \lim_{n\to \infty}\left(\mathcal{D}(u_2^n, \ell^{1/2}u_1^n, u_3^n)+\langle E\rangle(u_2^n,u_3^n, \widetilde{v}_2^n, \widetilde{v}_3^n)
-\mathbb{F}_p(\widetilde{v}_2^n, \widetilde{v}_3^n)\right).
\end{aligned}
\end{equation}
Since $\ell>1$, using \eqref{tech3Subb}, it follows that
\begin{equation}\label{caseia2}
\begin{aligned}
\mathcal{D}(u_2^n, \ell^{1/2}u_1^n, u_3^n)&=\mathcal{Q}(u_2^n, \ell^{1/2}u_1^n, u_3^n)-\mathbb{F}_p(u_2^n,u_3^n) \\
& \leq \ell \mathcal{Q}(u_2^n, u_1^n, u_3^n)-\mathbb{F}_p(u_2^n,u_3^n)\\
&= \mathcal{Q}(u_2^n, u_1^n, u_3^n)+\frac{T_1}{M_1}\mathcal{Q}(u_2^n, u_1^n, u_3^n)-\mathbb{F}_p(u_2^n,u_3^n).
\end{aligned}
\end{equation}
Substituting \eqref{caseia2} into \eqref{caseiaMI} and taking into account the
observation \eqref{tech3Suba},
it follows that
\begin{equation}\label{A9last}
\begin{aligned}
I^{(3)}_{M+T}& \leq \lim_{n\to \infty}\mathcal{I}(u_1^n,u_2^n,u_3^n)+\frac{T_1}{M_1}\frac{L_1}{T_1}+
\lim_{n\to \infty}\left ( \langle E\rangle(\widetilde{v}_2^n, \widetilde{v}_3^n)-\mathbb{F}_p(\widetilde{v}_2^n, \widetilde{v}_3^n)\right)\\
& < I^{(3)}_M+ \frac{L_2}{M_1}+
\lim_{n\to \infty}\left ( \langle E\rangle(\widetilde{v}_2^n, \widetilde{v}_3^n)-\mathbb{F}_p(\widetilde{v}_2^n, \widetilde{v}_3^n)\right)\\
& = I^{(3)}_M+\lim_{n\to \infty}\mathcal{I}(v_1^n, \widetilde{v}_2^n, \widetilde{v}_3^n)=I^{(3)}_M+I^{(3)}_T,
\end{aligned}
\end{equation}
which is the desired strict inequality. The same argument
applies in the case $L_1>L_2$ by switching indices and so we omit the details.
Assume now that $L_1=L_2$ and
consider the numbers
\[
\begin{aligned}
& \Pi_1=\frac{1}{M_2}\lim_{n\to \infty}\left(E(u_2^n)-\frac{1}{p} \int_{\mathbb{R}^{N}}(W\star |u_2^n|^p)|u_3^n|^p \ dx\right),\\
& \Pi_2=\frac{1}{T_2}\lim_{n\to \infty}\left(E(v_2^n)-\frac{1}{p} \int_{\mathbb{R}^{N}}(W\star|v_2^n|^p)|v_3^n|^p\ dx\right).
\end{aligned}
\]
We split the proof into two subcases: $\Pi_1\leq \Pi_2$ and $\Pi_1\geq \Pi_2.$
Since the proofs in both
subcases are similar, we only consider $L_1=L_2$ and $\Pi_1\leq \Pi_2.$
Let $F_n=(f_1^n, f_2^n, f_3^n)$, where
$f_1^n=\ell^{1/2}u_1^n$ with $\ell$ is defined as above, $f_2^n=s^{1/2}u_2^n$
with ${\displaystyle s=1+T_2/M_2}$, and $f_3^n$ is defined as in \eqref{ghDef}.
Since $s>1$,
using Lemma~\ref{tech1}, there exists $\delta>0$ such that
\begin{equation}\label{extraeqns3}
E(f_2^n)=E(s^{1/2}u_2^n)\leq sE(u_2^n)-\delta
\end{equation}
for all sufficiently large $n.$
Since $p\geq 2$, we have that $s^{p/2}\geq s>1.$ Using this fact, it is easy to check that
$\mathbb{F}_p(s^{1/2}f,g)\geq s\mathbb{F}_p(f,g)$ and
$\mathcal{Q}(s^{1/2}f,g,h)\leq \mathcal{Q}(f,g,h).$
Making use of these observations, \eqref{tech3Subb}, \eqref{extraeqns3}, and
taking into account
the definitions
${\displaystyle \ell=(M_1+T_1)/M_1}$ and ${\displaystyle s=(M_2+T_2)/M_2},$
we obtain that
\begin{equation}\label{caseia4ab}
\begin{aligned}
\mathcal{I}(F_n)& \leq \mathcal{D}(f_2^n, f_1^n, u_3^n)+\langle E\rangle(f_2^n, u_3^n, \widetilde{v}_3^n )\\
& \leq \langle E\rangle(f_2^n, u_3^n, \widetilde{v}_3^n )+\ell\mathcal{Q}(f_2^n, u_1^n, u_3^n)-\mathbb{F}_p(s^{1/2}u_2^n,u_3^n)\\
&\leq \langle E\rangle(f_2^n, u_3^n, \widetilde{v}_3^n )+\ell\mathcal{Q}(u_2^n, u_1^n, u_3^n)-s\mathbb{F}_p(u_2^n,u_3^n)\\
&\leq \mathcal{I}(u_1^n, u_2^n, u_3^n)+\frac{T_1}{M_1}\mathcal{Q}(u_2^n, u_1^n, u_3^n)+E(\widetilde{v}_3^n)\\
&\ \ \ + \frac{T_2}{M_2}\left(E(u_2^n)-\mathbb{F}_p(u_2^n,u_3^n) \right)-\delta
\end{aligned}
\end{equation}
Using this last estimate and making use of the assumptions $L_1=L_2$ and $\Pi_1\leq \Pi_2,$
we obtain that
\begin{equation}\label{caseia6}
\begin{aligned}
I^{(3)}_{M+T}& \leq \lim_{n\to \infty}\mathcal{I}(F_n)
 \leq I^{(3)}_M
+\frac{T_1}{M_1}\frac{L_1}{T_1}+\lim_{n\to \infty}E(v_3^n)+\frac{T_2}{M_2}(M_2\Pi_1)-\delta\\
&\leq I^{(3)}_M+\frac{L_2}{M_1}+\lim_{n\to \infty}E(v_3^n)+T_2\Pi_2-\delta=I^{(3)}_M+I^{(3)}_T-\delta,
\end{aligned}
\end{equation}
which gives the desired strict inequality.
\end{proof}
The next lemma establishes \eqref{MainSubadd} in the case $(A_3).$
\begin{lem}
For any $T\in \mathbb{R}_{+}^3$ and $M\in \mathbb{R}_{+}\times \{0\}\times \mathbb{R}_{+},$ one has
\[
I^{(3)}_{S}<I^{(3)}_M+I^{(3)}_M,\ \ S=M+T.
\]
\end{lem}
\begin{proof}
Let $\{(u_1^n, 0, u_3^n)\}_{n\geq 1}$ and $\{(v_1^n, v_2^n, v_3^n)\}_{n\geq 1}$
be minimizing sequences for $I^{(3)}_{M_1, 0, M_3}$ and $I^{(3)}_{T_1, T_2, T_3},$ respectively.
Define the real numbers
\[
G_1=T_1 \lim_{n\to \infty}\mathcal{Q}(v_2^n, u_1^n, u_3^n)\ \ \mathrm{and}\ \ G_2=M_1 \lim_{n\to \infty}\mathcal{Q}(v_2^n, v_1^n, v_3^n).
\]
Assume first that $G_1<G_2.$
Define $f_3^n$ as follows
\[
\widetilde{v}_3(\cdot)=v_3^n(\cdot + x_n \rho),\ f_3^n=u_3^n+\widetilde{v}_3^n,
\]
where $\rho$ is a unit vector in $\mathbb{R}^N$ ; and $x_n$ is chosen
such that $x_n\to 0$ as $n\to \infty$, and $u_3^n$ and $\widetilde{v}_3^n$ have disjoint supports.
Take $f_1^n=\ell^{1/2}u_1^n$ and $f_2^n=v_2^n,$ where $\ell=1+\frac{T_1}{M_1}.$ Let us write $F_n=(f_1^n, f_2^n, f_3^n)$.
Using the same argument as in \eqref{caseiaMI} and \eqref{caseia2}, we can obtain
\[
\begin{aligned}
\mathcal{I}(F_n)
& \leq \mathcal{I}(u_1^n, 0, u_3^n)+\langle E\rangle(v_2^n, \widetilde{v}_3^n)+\frac{T_1}{M_1}\mathcal{Q}(v_2^n, u_1^n, u_3^n)
-\mathbb{F}_p(v_2^n,\widetilde{v}_3^n).
\end{aligned}
\]
Using this estimate and the assumption $G_1<G_2$,
it follows that
\begin{equation}\label{CaseA31}
\begin{aligned}
I^{(3)}_S & \leq \lim_{n\to \infty}\mathcal{I}(f_1^n, f_2^n, f_3^n) \\
& \leq I^{(3)}_M+\frac{T_1}{M_1}\frac{G_1}{T_1}
 + \lim_{n\to \infty}\left( \langle E\rangle(v_2^n, \widetilde{v}_3^n)-\mathbb{F}_p(v_2^n,\widetilde{v}_3^n)\right)\\
 & < I^{(3)}_M+\frac{G_2}{M_1}+\lim_{n\to \infty}\left( \langle E\rangle(v_2^n, \widetilde{v}_3^n)-\mathbb{F}_p(v_2^n,\widetilde{v}_3^n)\right)\\
 & = I^{(3)}_M+\lim_{n\to \infty}\mathcal{I}(v_1^n, v_2^n, \widetilde{v}_3^n)=I^{(3)}_M+I^{(3)}_T,
\end{aligned}
\end{equation}
which is the desired strict inequality.
The proof for the case $G_1>G_2$ goes through unchanged and we do not repeat here.
Assume now that $G_1=G_2.$
As in the previous case, we consider the numbers
\[
\begin{aligned}
& \Gamma_1=\frac{1}{M_3}\lim_{n\to \infty}\left(E(u_3^n)-\frac{1}{p}\int_{\mathbb{R}^{N}}(W\star |u_3^n|^p)|v_2^n|^p\ dx \right),\\
& \Gamma_2=\frac{1}{T_3}\lim_{n\to \infty}\left(E(v_3^n)-\frac{1}{p}\int_{\mathbb{R}^{N}}(W\star |v_3^n|^p)|v_2^n|^p\ dx \right)
\end{aligned}
\]
and split the proof into two subcases: $\Gamma_1\leq \Gamma_2$ and $\Gamma_1\geq \Gamma_2.$
Consider the case that $G_1=G_2$ and $\Gamma_1\leq \Gamma_2.$ Take the functions
\[
f_1^n=\ell^{1/2}u_1^n,\ f_2^n=v_2^n,\ \mathrm{and}\ f_3^n=t^{1/2}u_3^n,
\]
where $\ell$ is defined as above and $t=(M_3+T_3)/M_3.$
Since $p\geq 2$ and $t>1$, we have that $t^{p/2}\geq t.$ Then it is straightforward
to see that $\mathbb{F}_p(f,t^{1/2}g)\geq t\mathbb{F}_p(f, g)$
and $\mathcal{Q}(f,g,t^{1/2}h)\geq \mathcal{Q}(f,g,h).$
Using these observations and \eqref{tech3Subb}, it follows that
\begin{equation}\label{caseibb1}
\begin{aligned}
\mathcal{I}(F_n) &=\mathcal{Q}(v_2^n, \ell^{1/2}u_1^n, t^{1/2}u_3^n)+\langle E\rangle(f_2^n,f_3^n)-\mathbb{F}_p(f_2^n, f_3^n)\\
& \leq \ell\mathcal{Q}(v_2^n, u_1^n, t^{1/2}u_3^n)+\langle E\rangle(f_2^n,f_3^n)-t\mathbb{F}_p(v_2^n, u_3^n)\\
& \leq \ell\mathcal{Q}(v_2^n, u_1^n, u_3^n)+\langle E\rangle(f_2^n,f_3^n)-t\mathbb{F}_p(v_2^n, u_3^n)
\end{aligned}
\end{equation}
Since $t>1,$ by an application of Lemma~\ref{tech1}, there exists $\delta>0$ such that for all sufficiently large $n,$ we have
\begin{equation}\label{caseibb2}
E(f_3^n)-t\mathbb{F}_p(u_2^n, v_3^n)\leq t\left(E(u_3^n)-\mathbb{F}_p(v_2^n, u_3^n) \right)-\delta.
\end{equation}
Using the definitions of $\ell$ and $t,$ it follows from \eqref{caseibb1} and \eqref{caseibb2} that
\[
\begin{aligned}
\mathcal{I}(F_n) \leq \mathcal{I}(u_1^n, 0, u_3^n)+E(v_2^n)+\frac{T_1}{M_1}\mathcal{Q}(v_2^n, u_1^n, u_3^n)
 +\frac{T_3}{M_3}\left(E(u_3^n)-\mathbb{F}_p(v_2^n, u_3^n) \right)-\delta
\end{aligned}
\]
Using the estimate above and the assumptions $G_1=G_2$, $\Gamma_1\leq \Gamma_2,$ it then follows that
\begin{equation}\label{Ficaseib}
\begin{aligned}
I^{(3)}_S& \leq \lim_{n\to \infty}\mathcal{I}(f_1^n, f_2^n, f_3^n)\\
& \leq  I^{(3)}_M + \lim_{n\to \infty}E(v_2^n)+\frac{T_1}{M_1}\frac{G_1}{T_1}+\frac{T_3}{M_3}(M_3\Gamma_1)-\delta \\
&\leq I^{(3)}_M + \lim_{n\to \infty}E(v_2^n)+\frac{G_2}{M_1}+T_3\Gamma_2-\delta\\
& = I^{(3)}_M+ \lim_{n\to \infty}\mathcal{I}(v_1^n, v_2^n, v_3^n)-\delta=I^{(3)}_M+I^{(3)}_T-\delta,
\end{aligned}
\end{equation}
which gives the desired strict inequality.
The proof in the case $G_1=G_2$ and $\Gamma_1\geq \Gamma_2$ is similar and we omit it.
\end{proof}
The next lemma establishes \eqref{MainSubadd} in the case $(B_3).$
\begin{lem}
For any $M\in \{0\}\times \mathbb{R}_{+}^2$ and $T\in \mathbb{R}_{+}\times \{0\}\times \mathbb{R}_{+},$ one has
\[
I^{(3)}_{M+T}<I^{(3)}_M+I^{(3)}_T.
\]
\end{lem}
\begin{proof}
Let $\{(0,u_2^n, u_3^n)\}_{n\geq 1}$ and $\{(v_1^n, 0, v_3^n)\}_{n\geq 1}$ be minimizing sequences for $I^{(3)}_M$ and
$I^{(3)}_T$ respectively.
Let $(C_1, C_2)\in \mathbb{R}^2$ be defined by
\[
C_1=T_3\lim_{n\to \infty}\mathcal{Q}(v_1^n, u_3^n, u_2^n)\ \ \mathrm{and}\ \ C_2=M_3\lim_{n\to \infty}\mathcal{Q}(v_1^n, v_3^n, u_2^n).
\]
We consider two cases: $C_1\leq C_2$ and $C_1\geq C_2.$ Suppose first that $C_1\leq C_2.$ Define $F_n=(f_1^n, f_2^n, f_3^n)$ as follows
\begin{equation}\label{caseiigAFd}
f_1^n=v_1^n,\ \ f_2^n=u_2^n,\ \ \mathrm{and}\ \ f_3^n=t^{1/2}u_3^n,
\end{equation}
where $t$ is defined as in the previous case.
Using Lemma~\ref{tech1}, there exists $\delta>0$ such that $E(f_3^n)\leq tE(u_3^n)-\delta$ for sufficiently large $n.$ Then,
by a direct computation and using the fact $t^{p/2}\geq t$, we obtain that
\begin{equation}\label{caseiigAIna}
\begin{aligned}
\mathcal{I}(F_n)& =\langle E\rangle(F_n)-t^{p/2}\mathbb{F}_p(u_3^n, v_1^n)-t^{p/2}\mathbb{F}_p(u_3^n, u_2^n)-\mathbb{F}_p(v_1^n, u_2^n)\\
&\leq \langle E\rangle(F_n)-t\mathbb{F}_p(u_3^n, v_1^n)-t\mathbb{F}_p(u_3^n, u_2^n)-\mathbb{F}_p(v_1^n, u_2^n)\\
& \leq \langle E\rangle(v_1^n, u_2^n)+t\mathcal{Q}(v_1^n, u_3^n, u_2^n)-\mathbb{F}_p(v_1^n, u_2^n)-\delta\\
& \leq \mathcal{I}(0, u_2^n, u_3^n)+E(v_1^n)+\frac{T_3}{M_3}\mathcal{Q}(v_1^n, u_3^n, u_2^n)-\mathbb{F}_p(v_1^n, u_2^n)-\delta
\end{aligned}
\end{equation}
Since ${\displaystyle \lim_{n\to \infty}\mathcal{Q}(v_1^n, u_3^n, u_2^n)=C_1/T_3, \lim_{n\to \infty}\mathcal{I}(0, u_2^n, u_3^n)=I^{(3)}_M},$
and $\mathbb{F}_p(v_1^n,u_2^n)\geq 0,$ it follows from
\eqref{caseiigAIna} that
\[
\begin{aligned}
I^{(3)}_{M+T}&\leq \lim_{n\to \infty}\mathcal{I}(f_n^n, f_2^n, f_3^n)\\
& \leq I^{(3)}_M+\lim_{n\to \infty}E(v_1^n)+\frac{T_3}{M_3}\frac{C_1}{T_3}-\delta\\
& \leq I^{(3)}_M+\lim_{n\to \infty}E(v_1^n)+\frac{C_2}{M_3}-\delta\\
& \leq I^{(3)}_M+\lim_{n\to \infty}\mathcal{I}(v_1^n, 0, v_3^n)-\delta=I^{(3)}_M+I^{(3)}_T-\delta,
\end{aligned}
\]
which gives the desired strict inequality.
The proof in the case $C_1\geq C_2$ is similar and we do not repeat here.
\end{proof}
The following lemma establishes \eqref{MainSubadd} in the case $(B_5).$

\begin{lem}
For any $M\in \{\mathbf{0}\}\times \mathbb{R}_{+}$ and $T\in \mathbb{R}_{+}^3,$ one has
\[
I^{(3)}_{M+T}<I^{(3)}_M+I^{(3)}_T.
\]
\end{lem}
\begin{proof}
Let $\{(0, 0, u_3^n)\}_{n\geq 1}$ and $\{(v_1^n, v_2^n, v_3^n)\}_{n\geq 1}$ be minimizing sequences
for $I^{(3)}_{0, 0, M_3}$ and $I^{(3)}_T$ respectively.
Let $(D_1, D_2)\in \mathbb{R}^2$ be defined as
\[
D_1=T_3\lim_{n\to \infty}\mathcal{Q}(v_1^n, u_3^n, v_2^n)\ \ \mathrm{and}\ \ D_2=M_3\lim_{n\to \infty}\mathcal{Q}(v_1^n, v_3^n, v_2^n).
\]
As before, we divide the proof into two cases $D_1\leq D_2$ and $D_1\geq D_2.$ In the first case $D_1\leq D_2,$ define
$F_n=(f_1^n, f_2^n, f_3^n)\in Y_3$ as follows
\[
f_1^n=v_1^n,\ f_2^n=v_2^n,\ \ \mathrm{and}\ \ f_3^n=t^{1/2}u_3^n,
\]
where $t$ is given by $\displaystyle{t=(M_3+T_3)/M_3}.$
Using Lemma~\ref{tech1}, there exists a number $\delta>0$ such that $E(f_3^n)\leq tE(u_3^n)-\delta$ for sufficiently large $n.$
Then, as in the previous case, it follows that
\[
\begin{aligned}
\mathcal{I}(F_n)& =\langle E\rangle(F_n)-t^{p/2}\mathbb{F}_p(u_3^n,v_1^n)-t^{p/2}\mathbb{F}_p(u_3^n,v_2^n)-\mathbb{F}_p(v_2^n, v_1^n)\\
& \leq \langle E\rangle(v_1^n, v_2^n)+t\mathcal{Q}(v_1^n,u_3^n, v_2^n)-\mathbb{F}_p(v_2^n, v_1^n)-\delta \\
&\leq \mathcal{I}(0, 0, u_3^n)+ \langle E\rangle(v_1^n, v_2^n)+ \frac{T_3}{M_3}\mathcal{Q}(v_1^n,u_3^n, v_2^n)-\delta.
\end{aligned}
\]
Since ${\displaystyle \mathcal{I}(0, 0, u_3^n)\to I^{(3)}_\tau}$ and $D_1\leq D_2$, it follows from the above estimate that
\[
\begin{aligned}
I^{(3)}_{M+T} & \leq \lim_{n\to \infty}\mathcal{I}(f_1^n, f_2^n, f_3^n)\\
&\leq I^{(3)}_M +\lim_{n\to \infty}\langle E\rangle(v_1^n, v_2^n)+\frac{T_3}{M_3}\frac{D_1}{T_3}-\delta \\
&\leq I^{(3)}_M +\lim_{n\to \infty}\langle E\rangle(v_1^n, v_2^n)+\frac{D_2}{M_3}-\delta \\
& = I^{(3)}_M+\lim_{n\to \infty}\mathcal{I}(v_1^n, v_2^n, v_3^n)-\delta=I^{(3)}_M+I^{(3)}_T-\delta,
\end{aligned}
\]
which gives the desired strict inequality.
The case $D_1\geq D_2$ uses the same argument and we do not repeat here.
\end{proof}

\begin{lem}
For any $M\in \{0\}\times \mathbb{R}_{+}\times \{0\}$ and $T\in \mathbb{R}_{+}\times \{0\}\times \mathbb{R}_{+},$ one has
\[
I^{(3)}_{M+T}<I^{(3)}_M+I^{(3)}_T.
\]
\end{lem}
\begin{proof}
Using Lemma~\ref{tech3}, let $\phi_{M_2}>0$ be such that
\[
E(\phi_{M_2})=\inf\left\{E(f):f\in H^1(\mathbb{R}^N)\ \mathrm{and}\ \|f\|_{L^2}^2=M_2 \right\}.
\]
Lemma~\ref{2MainlemEx} implies that there exist functions $\phi_{T_1}>0$ and $\phi_{T_3}>0$ such that
\[
\mathcal{I}(\phi_{T_1},\phi_{T_3})=\inf\left\{\mathcal{I}(f,g):f,g\in H^1(\mathbb{R}^N)\ \mathrm{and}\ \|f\|_{L^2}^2=T_1, \|g\|_{L^2}^2=T_3 \right\}.
\]
Clearly, we have that $\mathbb{F}_p(\phi_{M_2},\phi_{T_1})>0$ and $\mathbb{F}_p(\phi_{M_2},\phi_{T_3})>0$.
Then we obtain
\[
\begin{aligned}
I^{(3)}_{M+T} &\leq \mathcal{I}(\phi_{M_2},\phi_{T_1}, \phi_{T_3})\\
& = E(\phi_{M_2})+\mathcal{I}(\phi_{T_1},\phi_{T_3})-\mathbb{F}_p(\phi_{M_2},\phi_{T_1})-\mathbb{F}_p(\phi_{M_2},\phi_{T_3})\\
& = I^{(3)}_M+I^{(3)}_T-\mathbb{F}_p(\phi_{M_2},\phi_{T_1})-\mathbb{F}_p(\phi_{M_2},\phi_{T_3})<I^{(3)}_M+I^{(3)}_T,
\end{aligned}
\]
which is the desired strict inequality.
\end{proof}
We have now completed the proof of Lemma~\ref{MainSlemI}. The next lemma rules out the case of dichotomy.

\begin{lem}\label{dicMainm3}
Suppose that $\{(u_1^n,u_2^n, u_3^n)\}_{n\geq 1}\subset Y_3$ be any minimizing
sequence of $I^{(3)}_M$ and $Z^{(3)}$ be defined by \eqref{lamDef} with $m=3.$
Then, one has
\[
Z^{(3)}=M_1+M_2+M_3.
\]
\end{lem}
\begin{proof}
The proof goes through unchanged as in the proof of Lemma~\ref{dicMainm2} and we do not repeat here.
\end{proof}

\begin{lem}\label{3MainlemEx}
For every $M\in \mathbb{R}_{+}^3$, the set $\Lambda^{(3)}(M)$ is nonempty. Moreover, the following statements hold.

\smallskip

(i) For every $(\phi_1, \phi_2, \phi_3)\in \Lambda^{(3)}(M)$, there exists $\lambda_1$, $\lambda_2$, and $\lambda_3$ such that
\[
(\psi_1(x,t), \psi_2(x,t), \psi_3(x,t))=(e^{-i\lambda_1 t} \phi_1(x),e^{-i\lambda_2 t} \phi_2(x),e^{-i\lambda_3 t} \phi_3(x) )
\]
is a standing-wave solution of \eqref{NChoM2} with $m=3.$

\smallskip

(ii) The Lagrange multipliers $\lambda_1$, $\lambda_2$, and $\lambda_3$ satisfy $\lambda_j>0.$

\smallskip

(iii) For every $(\phi_1,\phi_2, \phi_3)\in \Lambda^{(3)}(M)$ there exists $\theta_j>0$ and real-valued functions
$\phi_{M_1}$, $\phi_{M_2}$, and $\phi_{M_3}$ such that
\[
\phi_{M_j}(x)>0\ \ \mathrm{and}\ \ \phi_j(x)=e^{i\theta_j}\phi_{M_j}(x),\ x\in \mathbb{R}^N.
\]
\end{lem}
\begin{proof}
The proof uses the same argument as in the proof of Lemma~\ref{2MainlemEx} and we omit the details.
\end{proof}

\section{Proof of main results}
We are now prepared to obtain our main results.

\smallskip

\textbf{Proof of Theorem~\ref{exisTHM}.} The proof follows from Lemmas~\ref{2MainlemEx} and \ref{2MainlemEx}.

\smallskip

\textbf{Proof of Theorem~\ref{StaTHM}.} Once we have obtained the relative compactness of minimizing sequences,
the proof of stability result uses a classical argument (\cite{CLi}) which we repeat here for the sake of completeness.
Suppose that $\Lambda^{(m)}(M)$ is not stable. Then there exist a
number $\epsilon >0,$ a sequence of times ${t_{n}},$ and a sequence $\{\psi_n(x,0)\}=
\{(\psi_1^n(x,0), \ldots, \psi_m^n(x,0))\}$ in $Y_m$ such that for all $n,$
\begin{equation}
\label{idconvtoF} \inf\{\|(\psi_1^n(x,0), \ldots, \psi_m^n(x,0))-\phi \|_{Y_m} : \phi \in
\Lambda^{(m)}(M)\}<\frac{1}{n};
\end{equation}
and
\begin{equation}\label{contra}
\inf\{\|(\psi_1^n(\cdot,t_n), \ldots,\psi_m^n(\cdot,t_n) )-\phi \|_{Y_m} : \phi \in
\Lambda^{(m)}(M)\} \geq \epsilon,
\end{equation}
for all $n,$ where $(\psi_1^n(x,t), \ldots,\psi_m^n(x,t) )$ solves \eqref{NChoM2} with initial data
$\psi_n(x,0).$
Since $\psi_n(x,0)$ converges to an element in $\Lambda^{(m)}(M)$ in $Y_m$ norm, and since
for $\phi \in \Lambda^{(m)}(M),$ we have
$\|\phi_j\|_{L^2}^2=M_j, 1\leq j\leq m$, and $\mathcal{I}(\phi)=I_M^{(m)},$
we therefore have
\[
\lim_{n\to \infty}\|\psi_j^n(x,0) \|_{L^2}^2=M_j,\ 1\leq j\leq m,\ \mathrm{and}\ \lim_{n\to \infty}\mathcal{I}(\psi_n(x,0))=I_{M}^{(m)}.
\]
Let us denote $\psi_j^n(\cdot,t_n)$ by $U_1^n$ for $1\leq j\leq m.$
We now choose $\{\alpha_j^n\}\subset \mathbb{R}^N$ such that
\[
\|\alpha_j^n\psi_j^n(x,0)\|_{L^2}^2=M_j,\ \ 1\leq j\leq m
\]
for all $n.$ Thus $\alpha_j^n \to 1$ for each $1\leq j\leq m.$
Hence the sequence $(f_1^n, \ldots, f_m^n)$ defined as $f_j^n=\alpha_j^nU_j^n$
satisfies
$\|f_j^n\|_{L^2}^2=M_j$ and
\begin{equation*}
\lim_{n\to \infty }\mathcal{I}(f_1^n, \ldots, f_m^n)= \lim_{n\to \infty }\mathcal{I}(\psi_n(\cdot, t_n))
 = \lim_{n\to \infty }\mathcal{I}(\psi_{n}(x, 0))=I_M^{(m)}.
\end{equation*}
Therefore $\{(f_1^n, \ldots, f_m^n)\}$ is a minimizing sequence for $I_M^{(m)}.$ From Theorem~\ref{exisTHM}, it follows that
for all $n$ sufficiently large, there exists $\phi_n \in \Lambda^{(m)}(M)$ such that
\[
\|(f_1^n, \ldots, f_m^n)-\phi_n\|_{Y_m} < \epsilon /2.
\]
But then we have
\begin{equation*}
\begin{aligned}
\epsilon & \leq \|\psi_{n}(\cdot, t_n)-\phi_n\|_{Y_m} \\
&\leq \|\psi_{n}(\cdot, t_n)-(f_1^n, \ldots, f_m^n) \|_{Y_m}
+ \|(f_1^n, \ldots, f_m^n) -\phi_n\|_{Y_m} \\
& \leq |1-\alpha_1^n|\cdot \|U_1^n\|_{H^1} +\ldots
+ |1-\alpha_m^n|\cdot \|U_m^n\|_{H^1} + \frac{\epsilon}{2}
\end{aligned}
\end{equation*}
and by taking $n \to \infty,$ we obtain that $\epsilon \leq \epsilon /2,$ a contradiction,
 and we conclude that $\Lambda^{(m)}(M)$ must in fact be stable.

\end{document}